\documentclass{amsart}
\usepackage[margin=1.0in]{geometry}
\usepackage{amsmath, amssymb, amsthm}
\usepackage{mathrsfs}
\usepackage{tikz}
\usepackage{xypic}
\usepackage{enumitem}
\usepackage{cleveref}

\newcommand{\Li}{\operatorname{Li}}
\newcommand{\ZZ}{\mathbb{Z}}
\newcommand{\RR}{\mathbb{R}}
\newcommand{\SL}{\operatorname{SL}}
\newcommand{\approxplus}{\operatorname{approx}^+}
\newcommand{\approxminus}{\operatorname{approx}^-}
\newcommand{\approxpm}{\operatorname{approx}^\pm}
\newcommand{\slope}{\operatorname{slope}}

\theoremstyle{plain}
\newtheorem{theorem}{Theorem}
\newtheorem{proposition}{Proposition}

\newtheorem{corollary}{Corollary}
\newtheorem{lemma}{Lemma}
\theoremstyle{definition}
\newtheorem{definition}{Definition}
\theoremstyle{remark}

\begin{document}

\title{The Three Gap Theorem, Interval Exchange Transformations, and Zippered Rectangles}
\author{Diaaeldin Taha}
\address{Department of Mathematics
University of Washington}
\email{dtaha@uw.edu}

\begin{abstract}
The Three Gap Theorem states that for any $\alpha \in (0,1)$ and any integer $N \geq 1$, the fractional parts of the sequence $0, \alpha, 2\alpha, \cdots, (N-1)\alpha$ partition the unit interval into $N$ subintervals having at most \emph{three} distinct lengths. We here provide a new proof of this theorem using zippered rectangles, and present a new gaps theorem (along with two proofs) for sequences generated as orbits of general interval exchange transformations.  We also derive a number of results on primitive points in lattices mirroring several properties of Farey fractions. This makes it possible to derive a previously known, explicit distribution result related to the Three Gap Theorem using ergodic theory.
\end{abstract}

\maketitle

\section{Introduction}
\subsection{Orbits, Gaps, and Randomness}
\label{subsection: orbits, gaps, and randomness}

A key theme in dynamical systems is understanding the extent to which orbits of a dynamical system resemble sequences of independent, identically distributed (i.i.d.) random variables. In this paper we consider sequences $s = (s_n)_{n=0}^\infty \subset [0,1)$ arising as orbits of rotations and interval exchange transformations $T:[0,1) \rightarrow [0, 1)$. That is, we consider sequences $s$ defined by
\begin{equation*}
    s_n = T^{n}(x),
\end{equation*}
with $T$ being an interval exchange transformation. In what follows, we use the interval $[0, 1)$, and the circle $\mathbb{S}^1$ interchangeably.

Following \cite{Marklof2006-tt} and later \cite{Athreya2012-kd}, the property of whether a sequence $s \subset [0, 1)$ equidistributes can be considered a first order statistical measure of randomness. Specifically, a sequence $s \subset [0, 1)$ is said to \textbf{equidistribute} if the measures $\Delta_N = \frac{0}{N-1}\sum_{n=1}^N \delta_{s_n}$ weak-$\ast$ converge to the Lebesgue measure $\operatorname{Leb}_{[0,1)}$ on the unit interval as $N \to \infty$. For the orbit sequences $s = \left(T^nx\right)_{n=0}^\infty$, the ergodicity of the map $T$ implies this convergence for almost every starting point $x$, and unique ergodicity strengthens this to every point. This covers the first order statistics of the sequences.

A finer measure of randomness is whether the \textbf{normalized gap distribution} for a sequence $(s_n)_{n=0}^\infty$ resembles that of an i.i.d sequence of uniformly distributed random variables $(X_n)_{n=0}^\infty \subset [0,1)$. More precisely, given a finite segment $(s_n)_{n=0}^{N-1}$, the points arrange themselves on the unit interval
\begin{equation*}
    0 \leq s_{\sigma_{s, N}(1)} \leq s_{\sigma_{s, N}(2)} \leq \cdots \leq s_{\sigma_{s, N}(N)} < 1,
\end{equation*}
with $\sigma_{s, N} : \{0, 1, \cdots, N-1\} \to \{0, 1, \cdots, N-1\}$ denoting the permutation of $\{0, 1, \cdots, N-1\}$ induced by the order of the points. The gaps
\begin{equation*}
    s_{\sigma_{s,N}(1)} - 0,\ s_{\sigma_{s,N}(2)} - s_{\sigma_{s,N}(1)},\ \cdots,\ s_{\sigma_{s,N}(N)} - s_{\sigma_{s,N}(N-1)}, 1 - s_{\sigma_{s,N}(N)}
\end{equation*}
form a multiset $\widetilde{\operatorname{Gaps}}_{s, N}$. The limiting behavior of the normalized gap distribution \[\lim_{N \to \infty} \frac{\#\left((N \cdot \widetilde{\operatorname{Gaps}}_{s, N}) \cap (a, b)\right)}{N}\] for $0 \leq a  < b \leq \infty$ gives the second order statistics of the sequence.

For an i.i.d sequence of uniformly distributed variables $(X_n)_{n=0}^\infty$, the limiting behavior is Poissonian: for any $t>0$,
\begin{equation*}
    \lim_{N \to \infty} \frac{\#\left((N \cdot \widetilde{Gaps}_{(X_n)_{n=0}^\infty, N} \cap (t, \infty))\right)}{N} = e^{-t}.
\end{equation*}
Sequences $s$ whose gaps stray from that Poissonian limiting behavior are called \textbf{exotic} in \cite{Athreya2012-kd}. The interested reader should refer to \cite{Athreya2012-kd} for examples of exotic sequences, and to get an understanding of the dynamical approach to gap distributions in general.

Circle rotations and interval exchange transformations (IETs) are low complexity maps (having zero topological entropy) essential to the study of polygonal billiards and linear flows on translation surfaces. In this paper, we study gaps for sequences $s$ that arise as orbits of circle rotations $s = \left(R_\alpha^n 0\right)_{n=0}^\infty$, and interval exchange transformations $s = \left(T^n0\right)_{n=0}^\infty$.

\subsection{Goals and Organization}

We have three goals in this paper: interpret and prove the Three Gap Theorem (\cref{theorem: Three Gap Theorem}) geometrically using zippered rectangle decompositions, draw parallels between Farey fractions and primitive points in arbitrary lattices towards deriving a gap distribution result related to the Three Gap Theorem (\cref{theorem: continuous distribution arising from the three gap theorem}) using ergodic theory, and generalize the Three Gap Theorem to $d$-IETs. We organize the paper as follows.

\begin{itemize}
\item In the remainder of this section, we state the well-known Three Gap Theorem (\cref{theorem: Three Gap Theorem}), our generalization thereof to $d$-IETs (\cref{theorem: d + 2 gap theorem}), an average gap distribution result from \cite{Polanco_undated-ts} related to the Three Gap Theorem (\cref{theorem: continuous distribution arising from the three gap theorem}), and our interepretition of the aforementioned gap distribution in terms of the average height of zippered rectangles in the space of unimodular tori (\cref{theorem: distribution for TGT}). We also axiomatize a gap distribution result for our generalized gap theorem for $d$-IETS (\cref{theorem: axiomatic theorem}). We start by presenting the definition of IETs, and follow that with a short exposition of the Three Gap Theorem to provide historical context.
\item In \cref{section: background} we review some facts about the space of unimodular lattices/tori, the horocycle and geodesic flows, and zippered rectangle decompositions.
\item In \cref{section: the three gap theorem and zippered rectangles} we relate gaps for orbits of circle rotations to zippered rectangle decompositions of unit area tori, providing proofs for \cref{theorem: Three Gap Theorem}, and \cref{theorem: distribution for TGT} along the way.
\item In \cref{section: primitive lattice points}, we build the tools to explicitly derive the gap distribution in \cref{theorem: distribution for TGT} (i.e. derive the expression in \cref{theorem: continuous distribution arising from the three gap theorem}) for all tori. In particular, we show that primitive lattice points share many properties with Farey fractions: Farey neighbors (\cref{lemma: Farey Neighbors}), dynamics of the Farey triangle and the BCZ map (\cref{corollary: primitive lattice points and the Farey triangle}), and generative algorithm (\cref{theorem: generating primitive lattice points}). We also derive the zippered rectangle decompositions of  tori whose corresponding lattices have no short horizontal vectors (\cref{corollary: Zippered Rectangle Decompositions of Tori}), and extend previously known distribution results for the denominators of Farey fractions to the $x$-components of primitive vectors of all lattices (\cref{theorem: equidistribution of x-component measures}). This makes it possible to derive the explicit expression for the limiting distribution in \cref{theorem: continuous distribution arising from the three gap theorem} using geometry and dynamics.
\item In \cref{section: gap theorem for d-iets via graph theory}, we present a second proof of the $d+2$ Gap Theorem (\cref{theorem: d + 2 gap theorem}) using a purely combinatorial argument.
\end{itemize}

\subsection{Interval Exchange Transformations}
\label{subsection: interval exchange transformations}

Given $\lambda = (\lambda_1, \lambda_2, \cdots, \lambda_d) \in \RR^n$ with $\lambda_i \geq 0$ and $|\lambda| = \sum_{i=1}^d \lambda_i = 1$, and a permutation $\pi \in S_d$ on the $d$ letters $\{1, 2, \cdots, d\}$, we can define a map $T = T_{(\lambda, \pi)} : [0, 1) \to [0, 1)$ exchanging the intervals $I_i = [\sum_{k=1}^{i-1}\lambda_k,\sum_{k=1}^i \lambda_k)$ ($i=1, 2, \cdots, d$) according to the permutation $\pi$. That is, if $x \in I_i$, then \[T_{(\lambda, \pi)}(x) = x - \sum_{k < i}\lambda_k + \sum_{\pi(k) < \pi(i)}\lambda_k.\] We say that $T_{(\lambda, \pi)}$ is a \textbf{$d$-interval exchange transformation}, or \textbf{$d$-IET} for short, with \textbf{length data} $\lambda$ and \textbf{combinatorial data} $\pi$.

We say that a permutation $\pi \in S_d$ is \textbf{irreducible}, and denote that by $\pi \in S_d^o$, if $\pi(\{1, 2, \cdots, k\}) \neq \{1, 2, \cdots, k\}$ for every $k < d$.

The length data $\lambda$ can be parametrized by the unit simplex $\Delta^{d-1} := \{(t_1, t_2, \cdots, t_d) \in \RR^d \mid t_i \geq 0, \text{ and } \sum_{i=1}^d t_i = 1\}$. The unit simplex comes with the Lebesgue measure $\operatorname{Leb}_{\Delta^{d-1}}$, which makes it possible to talk about ``almost all $d$-IETs''.

Denote the discontinuities of $T^{-1}$ by $\alpha_0 = 0, \alpha_1, \cdots, \alpha_d = 1$, and those of $T$ by $\beta_0 = 0, \beta_1, \cdots, \beta_d = 1$. Note that the subintervals $(\alpha_{i-1}, \alpha_i)$ are permuted by $T^{-1}$, and get sent to the subintervals $(\beta_{\pi^{-1}(i) - 1}, \beta_{\pi^{-1}(i)})$.

The IET $T$ is said to satisfy the \textbf{infinite distinct orbit condition} (i.d.o.c for short) or \textbf{Keane's condition} if the orbits $\{(T^{-1})^n\alpha_1\}_{n=1}, \{(T^{-1})^n\alpha_2\}_{n=1}, \cdots, \{(T^{-1})^n\alpha_{d-1}\}_{n=1}$ are both infinite and pairwise distinct. The i.d.o.c property is the correct notion of irrationality for IETs since Keane \cite{Keane1975-ai} showed that an IET with i.d.o.c is minimal, and if an IET has rationally independent length data $\lambda$, then it satisfies the i.d.o.c.

For more on IETs, the interested reader can refer to the excellent survery \cite{Yoccoz2007-zz}.

\subsubsection{Gaps of IETs}
\label{subsubsection: gaps of iets}

Let $T = T_{(\pi, \lambda)} : [0, 1) \to [0, 1)$ be a $d$-interval exchange map with irreducible combinatorial data $\pi \in S_d^o$. For an integer $N \geq 1$, consider the orbit segment $\{T^n0\}_{n=0}^{N-1}$, ordered on $[0, 1)$ as \[0 = T^{\sigma_{T, N}(0)}0 < T^{\sigma_{T, N}(1)}0 < \cdots < T^{\sigma_{T, N}(N-1)}0 < 1\] with $\sigma_{T, N} : \{0, 1, \cdots, N-1\} \to \{0, 1, \cdots, N-1\}$ the permutation induced by the order the points arrange themselves on the interval.

We denote for $i = 0, 1, \cdots, N-1$ the $i$th gap for the orbit segment $(T^n0)_{n=0}^{N-1}$ by
\begin{equation*}
    \operatorname{gap}_{T, N}(i) :=
    \begin{cases}
    T^{\sigma_{T, N}(i + 1)}0 - T^{\sigma_{T, N}(i)}, & i \neq N - 1 \\
    1 - T^{\sigma_{T, N}(N-1)}0, & i = N - 1
    \end{cases}.
\end{equation*}
We also denote the set of gap lengths by 
\begin{equation*}
    \operatorname{Gaps}_{T, N} := \{\operatorname{gap}_{T, N}(i) \mid 0 \leq i < N\},
\end{equation*}
and the multiset of gaps by
\begin{equation*}
    \widetilde{\operatorname{Gaps}}_{T, N} := \{\!\{\operatorname{gap}_{T, N}(i) \mid 0 \leq i < N\}\!\}.
\end{equation*}
We identify the elements of the multiset of gaps $\widetilde{\operatorname{Gaps}}_{T, N}$ with the actual gaps. That is, the set of open intervals $\left\{(T^{\sigma_{T, N}(0)}0, T^{\sigma_{T, N}(1)}0), (T^{\sigma_{T, N}(1)}0, T^{\sigma_{T, N}(2)}0), \cdots, (T^{\sigma_{T, N}(N-1)}0, 1) \right\}$ bounded by the orbit points $\{T^n0\}_{n=0}^{N-1}$.

\subsection{Circle Rotation}

For $\alpha \in [0, 1)$, denote the circle rotation map $t \mapsto t + \alpha \mod 1$ by $R_\alpha$. The circle rotation is a $2$-IET with length data $\lambda = (1 - \alpha, \alpha)$, and combinatorial data $\pi = (1\ 2)$. Also, for $t \in \RR$, let $\{t\} = t - \lfloor t \rfloor$ denote the fractional part of $t$. The orbit sequence $\left(R_\alpha^n0\right)_{n=0}^\infty$ and Kronecker's sequence of fractional parts $\left(\{n\alpha\}\right)_{n=0}^\infty$ agree.

For $T = R_\alpha$, the permutation $\sigma_{R_\alpha, N}$, the set of gap lengths $\operatorname{Gaps}_{T, N}$, and the multiset of gaps $\widetilde{\operatorname{Gaps}_{T, N}}$ are very well understood as can be seen in \cref{subsubsection: the combinatorial structure of the gaps of circle rotations} and \cref{subsubsection: the statistics of the gaps of circle rotations}.

\subsubsection{The Combinatorial Structure of $\widetilde{\operatorname{Gaps}}_{R_\alpha, N}$}
\label{subsubsection: the combinatorial structure of the gaps of circle rotations}

It was observed by H. Steinhaus (1957) that any finite segment of the orbit sequence $\left(R_\alpha^n0\right)_{n=0}^{N-1}$ induces three gap lengths. The first proof of the \textbf{Steinhaus Conjecture} was given by V. S\'{o}s \cite{Sos1957-nu, Sos1958-em}. This result is now commonly known as the \textbf{Three Gap Theorem}.

\begin{theorem}[Three Gap Theorem]
\label{theorem: Three Gap Theorem}
For any $\alpha \in (0, 1)$, and any integer $N \geq 1$,
\begin{equation*}
    \#\operatorname{Gaps}_{R_\alpha, N} \leq 3
\end{equation*}
\end{theorem}
 
\begin{figure}
\centering
\begin{tikzpicture}[scale = 0.75]
\draw (0,0) circle (3);
\draw

node[circle, fill=blue, scale=0.5] at (0,3) {}
node at (0, 3.5) {$\scriptscriptstyle R_\alpha^0(0)$}

node[circle, fill=blue, scale=0.5] at (-2.89171, -0.798766) {}
node at (-3.37366, -0.931894) {$\scriptscriptstyle R_\alpha^1(0)$}

node[circle, fill=blue, scale=0.5] at (1.53987, -2.57465) {}
node at (1.79651, -3.00376) {$\scriptscriptstyle R_\alpha^2(0)$}

node[circle, fill=blue, scale=0.5] at (2.07171, 2.16979) {}
node at (2.417, 2.53143) {$\scriptscriptstyle R_\alpha^3(0)$}

node[circle, fill=blue, scale=0.5] at (-2.64307, 1.41921) {}
node at (-3.08359, 1.65575) {$\scriptscriptstyle R_\alpha^4(0)$}

node[circle, fill=blue, scale=0.5] at (-0.664248, -2.92554) {}
node at (-0.774955, -3.41313) {$\scriptscriptstyle R_\alpha^5(0)$}

node[circle, fill=blue, scale=0.5] at (2.99679, 0.13867) {}
node at (3.49626, 0.161782) {$\scriptscriptstyle R_\alpha^6(0)$}

node[circle, fill=blue, scale=0.5] at (-0.931577, 2.8517) {}
node at (-1.08684, 3.32698) {$\scriptscriptstyle R_\alpha^7(0)$}

node[circle, fill=blue, scale=0.5] at (-2.50072, -1.65723) {}
node at (-2.91751, -1.93343) {$\scriptscriptstyle R_\alpha^8(0)$}

node at (0.393103, -2.4689) {A}
node at (2.27511, 1.03627) {A}
node at (-1.60463, 1.91708) {A}
node at (-1.42063, -2.05713) {A}
node at (-2.48445, 0.278445) {A}
node at (0.929891, 2.32063) {A}

node at (-0.393045, 2.46891) {C}
node at (-2.27514, -1.03622) {C}

node at (2.20253, -1.18274) {B};
\end{tikzpicture}
\caption{The orbit segment $\{R_\alpha^n(0)\}_{n=0}^8$ for $\alpha = 1/\sqrt{2}$. In this case, $A = \frac{3 \sqrt{2}}{2} - 2$, $B = A + C = 3 - \frac{4 \sqrt{2}}{2}$, and 
$C = 5 - \frac{7 \sqrt{2}}{2}$.}
\end{figure}
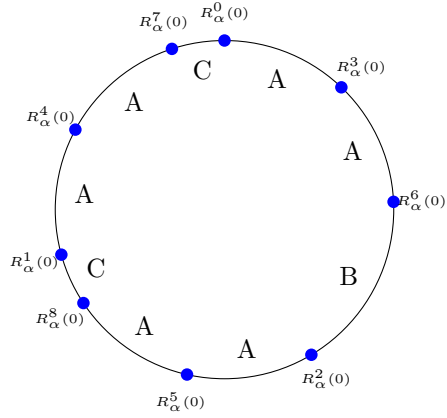

The list of mathematicians who subsequently proved the Three Gap Theorem includes S. Swierczkowski \cite{Swierczkowski1958-ji}, Sz\"{u}z, Erd\"{o}s, Turan, Chung \& Graham \cite{Chung1976-tu} among others. Recent work on the Three Gaps Theorem and how it relates to the spaces of lattices has been done by Marklof \& Strömbergsson in \cite{Marklof2016-tx}, and by Haynes \& Marklof in \cite{Haynes2017-xs}. A formal proof using the Coq proof assistant system has been published by Mayero \cite{Mayero2006-dy}. We present our proof in \cref{section: the three gap theorem and zippered rectangles}.

The different proofs usually come with a description of the gap combinatorial structure, giving the following theorem. (Check \cite{Polanco_undated-ts} for a proof.) We geometrically interpret and prove the same theorem (minus the permutation) using zippered rectangles in \cref{corollary: correspondence}, and \cref{corollary: Zippered Rectangle Decompositions of Tori} of this paper.

\begin{theorem}[Combinatorial Structure of $\widetilde{\operatorname{Gaps}}_{R_\alpha, N}$]
\label{theorem: combinatorial structure theorem}
Consider $\alpha \in (0, 1)$, and $N \geq 1$ an integer. Let $\mathcal{F}(N) = \{0 \leq \frac{a}{q} \leq 1 \mid \gcd(a, q) = 1,\, q \leq N\}$ be the Farey fractions of order $N$. If $\alpha = \frac{a}{q} \in \mathcal{F}(N)$, then $\widetilde{\operatorname{Gaps}}_{R_\alpha, N}$ has exactly $q$ elements, each of length $\frac{1}{q}$. Otherwise, there exists consecutive fractions $\frac{a_1}{q_1}$ and $\frac{a_2}{q_2}$ in $\mathcal{F}(N)$ such that $\frac{a_1}{q_1} < \alpha < \frac{a_2}{q_2}$. In that case, $\widetilde{\operatorname{Gaps}}_{R_\alpha, N}$ contains\footnote{The number of one of the gaps might be $0$. For instance, the number of $C$ gaps is zero on the Farey arc $\left(\frac{0}{1}, \frac{1}{N}\right)$.}
\begin{center}
\begin{tabular}{c c c}
$N - q_1$ gaps & & $A = q_1 \alpha - a_1 = \{q_1 \alpha\}$ \\
$q_1 + q_2 - N$ gaps & of length & $B = A + C$\\
$N - q_2$ gaps & & $C = a_2 - q_2 \alpha  = 1 - \{q_2 \alpha\}$
\end{tabular}
\end{center}

In the second case, the permutation $\sigma = \sigma_{\alpha, N} : \{0, 1, \cdots, N-1\} \to \{0, 1, \cdots, N-1\}$ giving the ordering \[0 = \{\sigma_{\alpha, N}(0) \alpha\} < \{\sigma_{\alpha, N}(1) \alpha\} < \cdots < \{\sigma_{\alpha, N}(N-1)\alpha\} < 1\] is recursively defined by
\[\sigma_0 = 0,\] and
\begin{equation*}
\sigma_{i+1} - \sigma_i =
\begin{cases}
q_1, & \text{if } \sigma_i \in [0, N-q_1) \text{ ($A$ gap)}\\
q_1 - q_2, & \text{if } \sigma_i \in [N-q_1, q_2) \text{ ($B$ gap)} \\
-q_2, & \text{if } \sigma_i \in [q_2, N) \text{ ($C$ gap)}
\end{cases}
.
\end{equation*}
\end{theorem}

Note that in the second case, when $\alpha \not\in \mathcal{F}(N)$, the elements belonging to the same Farey arc $\left(\frac{a_i}{q_i}, \frac{a_{i+1}}{q_{i+1}}\right)$ in $\mathcal{F}(N)$ have the same gap structure. The gap lengths vary linearly between the two end points of the arc. Finally, the elements of $\mathcal{F}(N)$ give the \emph{non-generic} gap structure, and are pretty much handled by modular arithmetic.

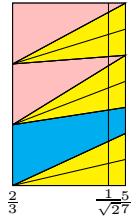
\begin{figure}
\begin{center}
\begin{tikzpicture}[scale = 1.5]
\draw (0, 0) -- (1, 0) -- (1, {(1 + sqrt(5))/2}) -- (0, {(1 + sqrt(5))/2}) -- (0, 0);
\draw[fill = yellow] (0, {0 * (1 + sqrt(5))/(2 * 3)}) -- (1, {0 * (1 + sqrt(5))/(2 * 7)}) -- (1, {2 * (1 + sqrt(5))/(2 * 7)}) -- (0, 0);
\draw[fill = cyan] (0, {0 * (1 + sqrt(5))/(2 * 3)}) -- (1, {2 * (1 + sqrt(5))/(2 * 7)}) -- (1, {3 * (1 + sqrt(5))/(2 * 7)}) -- (0, {1 * (1 + sqrt(5))/(2 * 3)}) -- (0, 0);
\draw[fill = yellow] (0, {1 * (1 + sqrt(5))/(2 * 3)}) -- (1, {3 * (1 + sqrt(5))/(2 * 7)}) -- (1, {5 * (1 + sqrt(5))/(2 * 7)}) -- (0, {1 * (1 + sqrt(5))/(2 * 3)});
\draw[fill=pink] (0, {1 * (1 + sqrt(5))/(2 * 3)}) -- (1, {5 * (1 + sqrt(5))/(2 * 7)}) -- (0, {2 * (1 + sqrt(5))/(2 * 3)}) -- (0, {1 * (1 + sqrt(5))/(2 * 3)});
\draw[fill = yellow] (0, {2 * (1 + sqrt(5))/(2 * 3)}) -- (1, {5 * (1 + sqrt(5))/(2 * 7)}) -- (1, {7 * (1 + sqrt(5))/(2 * 7)}) -- (0, {2 * (1 + sqrt(5))/(2 * 3)});
\draw[fill=pink] (0, {2 * (1 + sqrt(5))/(2 * 3)}) -- (1, {7 * (1 + sqrt(5))/(2 * 7)}) -- (0, {3 * (1 + sqrt(5))/(2 * 3)}) -- (0, {2 * (1 + sqrt(5))/(2 * 3)});
\draw (0, 0) -- (1, {(1 + sqrt(5))/(2 * 7)});
\draw (0, 0) -- (1, {2 * (1 + sqrt(5))/(2 * 7)});
\draw (0, {1 * (1 + sqrt(5))/(2 * 3)}) -- (1, {3 * (1 + sqrt(5))/(2 * 7)});
\draw (0, {1 * (1 + sqrt(5))/(2 * 3)}) -- (1, {4 * (1 + sqrt(5))/(2 * 7)});
\draw (0, {1 * (1 + sqrt(5))/(2 * 3)}) -- (1, {5 * (1 + sqrt(5))/(2 * 7)});
\draw (0, {2 * (1 + sqrt(5))/(2 * 3)}) -- (1, {5 * (1 + sqrt(5))/(2 * 7)});
\draw (0, {2 * (1 + sqrt(5))/(2 * 3)}) -- (1, {6 * (1 + sqrt(5))/(2 * 7)});
\draw (0, {2 * (1 + sqrt(5))/(2 * 3)}) -- (1, {7 * (1 + sqrt(5))/(2 * 7)});
\draw ({21 * (1/sqrt(2) - 2/3)}, 0) -- ({21 * (1/sqrt(2) - 2/3)}, {(1+sqrt(5))/2});
\draw node at (0, -0.15) {$\scriptstyle \frac{2}{3}$} node at ({21 * (1/sqrt(2) - 2/3)}, -0.15) {$\scriptstyle \frac{1}{\sqrt{2}}$} node at (1, -0.15) {$\scriptstyle \frac{5}{7}$};
\end{tikzpicture}
\end{center}
\caption{The Farey arc $(2/3, 5/7)$, with $A$ gaps in \textcolor{yellow}{yellow}, $B$ gaps in \textcolor{cyan}{cyan}, and $C$ gaps in \textcolor{pink}{pink}.}
\end{figure}

In the following theorem, we generalize the Three Gap Theorem to general $d$-IETs.

\begin{theorem}[$d+2$ Gap Theorem]
\label{theorem: d + 2 gap theorem}
Let $T$ be a $d$-IET with combinatorial data $\pi$ satisfying the i.d.o.c. For any integer $N\geq 1$,
\begin{itemize}
\item if $\pi^{-1}(\pi(1) - 1) = d$, then $\# \operatorname{Gaps}_{T, N} \leq d+1$, and 
\item if $\pi^{-1}(\pi(1) - 1) \neq d$, then $\# \operatorname{Gaps}_{T, N} \leq d+2$.
\end{itemize}
\end{theorem}

The bounds in theorem \cref{theorem: d + 2 gap theorem} are \emph{optimal}, improving a previous $3(d-1)$ bound that could be obtained as a consequence of Boshernitzan graph theoritic approach in \cite{Boshernitzan1985-fb} which we present in \cref{subsection: the Rauzy graph approach}. We give \emph{two} proofs of the above gap theorem: In \cref{section: gap theorem for d-iets via zippered rectangles}, we show that our zippered rectangles proof of the Three Gap Theorem extends naturally to suspensions over IETs, thus proving \cref{theorem: d + 2 gap theorem}. In \cref{section: gap theorem for d-iets via graph theory}, we show how the Rauzy graph approach can be modified to prove \cref{theorem: d + 2 gap theorem}.

\subsubsection{The Statistics of $\widetilde{\operatorname{Gaps}}_{\alpha, N}$}
\label{subsubsection: the statistics of the gaps of circle rotations}

Fix $\alpha \in (0,1)$, and $N \geq 1$ an integer. The expected value of the gap lengths is \[\frac{1}{N} \sum_{i=0}^{N-1} \operatorname{gap}_{R_\alpha, N}(i) = \frac{1}{N}.\] Let
\begin{eqnarray*}
\mathcal{G}_{R_\alpha, N}(z) &:=& \# \{\ell \in \widetilde{\operatorname{Gaps}}_{R_\alpha, N} \mid \ell \geq \frac{z}{N}\} \\
&=& \#\{L \in N \cdot \widetilde{\operatorname{Gaps}}_{R_\alpha, N} \mid L \geq z\}
\end{eqnarray*}
denote the number of \textbf{normalized gap lengths} greater than or equal to $z$.

Now, for $0 < a < b < 1$, consider the \textbf{average gap distribution}
\[g_{R_\cdot}^{[a, b]}(z; N) := \frac{1}{b-a} \int_a^b \frac{\mathcal{G}_{R_\alpha, N}(z)}{N}\, d\alpha.\] It is proved in \cite{Polanco_undated-ts} that $g_{R_\cdot}^{[a, b]}(z; N)$ has a limit.

\begin{theorem}[\cite{Polanco_undated-ts}]
\label{theorem: continuous distribution arising from the three gap theorem}
As $N \to \infty$, the limit $\lim_{N \to \infty} g_{R_\cdot}^{[a, b]}(z; N)$ exists, is independent of $[a, b]$, and is equal to
\begin{equation*}
    g(z)
    = \frac{6}{\pi^2} \scriptsize
        \begin{cases}
        \frac{\pi^2}{6} - z & 0 < z < 1 \\
        \begin{aligned}
        \log^2(2) - \frac{2\pi^2}{3} - 1 + \left(\frac{z}{2} - \frac{2}{z}\right)\log\left(\frac{2-z}{z - 1}\right) + \frac{3z}{2}\log\left(\frac{z}{z - 1}\right) \\
        -\log\left(\frac{4}{z}\right)\log(z) + 4\Li_2(\frac{1}{z}) + 2\Li_2(\frac{z}{2})
        \end{aligned} & 1 < z < 2 \\
        -1 + \left(\frac{z}{2} - \frac{2}{z}\right) \log \left(\frac{z - 2}{z - 1}\right) + \frac{3}{2} z \log \left(\frac{z}{z - 1}\right) + 4 \Li_2\left(\frac{1}{z}\right) - 2 \Li_2\left(\frac{2}{z}\right) & 2 < z
        \end{cases}
\end{equation*}
where the Dilogarithm is defined for $|z| \leq 1$ by \[\Li_2(z) := \sum_{n=1}^\infty \frac{z^n}{n^2}.\]
\end{theorem}

In \cref{section: the three gap theorem and zippered rectangles}, we interpret the average gap distribution $g_{R_\cdot}^{[a, b]}(z; N)$ as the average height of zippered rectangles, and prove the existence of the limit as $N \to \infty$ using the equidistribution of large horocycles. In \cref{section: primitive lattice points}, we derive the explicit distribution using limiting distributions of primitive lattice points. 

The approach from \cref{section: the three gap theorem and zippered rectangles} easily extends to prove the following theorem. (Refer to \cref{subsection: the space of unimodular lattices} for the definitions of $X_2$ and $\mu_2$, and to \cref{subsection: notation}, and \cref{subsection: gap distribution for circle rotations} for the definitions of $f_z$ and $\mathfrak{Z}$.)

\begin{theorem}
\label{theorem: distribution for TGT}
The average gap distribution $g_{R_\cdot}^{[a, b]}(z; N)$ has a limiting distribution
\begin{equation*}
    \lim_{N \to \infty} g_{R_\cdot}^{[a, b]}(z; N) = \int_{X_2} f_z(\mathfrak{Z}(\Lambda))\, d\mu_2(\Lambda).
\end{equation*} That is, $g_{R_\cdot}^{[a, b]}(z; N)$ in the limit is equal to the distribution of the heights of the zippered rectangle decomposition of unit area tori.
\end{theorem}

In \cref{subsection: gap distribution for IETs}, we show how the same approach can be axiomatized to cover compositions of a fixed IET and circle rotations. (Refer to \cref{subsection: gap distribution for IETs} for the definition of $g_{T \circ R}^{[a, b]}(z; N)$, to \cref{subsection: the space of unimodular lattices} for the definitions of $g_t$ and $h_\cdot$, and \cref{subsection: the surface S_T defined by IET T} for the definition of $S_T$, and to \cref{subsection: notation}, \cref{subsection: gap distribution for circle rotations}, and \cref{subsection: gap distribution for IETs} for the definitions of $f_z$ and $\mathfrak{Z}_\text{canon}$.)

\begin{theorem}
\label{theorem: axiomatic theorem}
Let $T$ be an interval exchange transformation, with $S_T \in \mathcal{H}$. If an equidistribution result of large horocycles result implying $(g_t)_\ast(\pi_{S_T})(h_\cdot)_\ast\operatorname{Unif}_{[a, b]} \rightharpoonup \mu$, with $\mu$ a measure on $\mathcal{H}$, holds true, then the average gap distribution $g_{T \circ R}^{[a, b]}(z; N)$ has a limiting distribution
\begin{equation*}
    \lim_{N \to \infty} g_{T \circ R_\cdot}^{[a, b]}(z; N) = \int_\mathcal{H} f_z(\mathfrak{Z}_\text{canon}(S))\,d\mu(S).
\end{equation*}
\end{theorem}

\section{Background}
\label{section: background}

\subsection{Linear Flow on Translation Surfaces}
\label{subsection: linear flow on translation surfaces}

The relationship between linear flows on tori (and translation surfaces in general) and interval exchange transformations is widely known.

\begin{proposition}
Let $S$ be an open bounded geodesic segment in a translation surface. The first return map of a linear flow on the translation surface to $S$ is an interval exchange map.
\end{proposition}

In particular, the following proposition is true. (The author has not been able to find an appearance of this in the literature preceeding \cite{Arnold2007-kz}).

\begin{proposition}[Three Interval Theorem]
\label{proposition: Three Interval Theorem}
The first return map defined by a minimal linear flow on a torus to a transversal to the flow gives an interval exchange map with $2$ or $3$ intervals.
\end{proposition}

In \cref{section: the three gap theorem and zippered rectangles}, we show that the above proposition is equivalent to the Three Gap Theorem (\cref{theorem: Three Gap Theorem}), and thus provides a geometric proof of the aforementioned result.

The same phenomenon is true for general translation surfaces. (Check \cite{Yoccoz2007-zz} or \cite{Zorich2006-of}.)

\begin{proposition}
\label{proposition: number of discontinuities of first return map}
Let $S$ be a translation surface, and denote its genus and number of singularities by $g$ and $s$ respectively. Write $d_S = 2g + s - 1$. Then the first return map of a minimal linear flow on $S$ to a transversal $X$ is an interval exchange interval $T$ on 
\begin{itemize}
\item $d_T = d_S$ intervals if the forward or backward trajectories of both end points of $X$ under the flow hit singularities,
\item $d_T = d_S + 1$ intervals if the forward or backward trajectories of exactly one of the end points of $X$ under the flow hits a singularity
\item $d_T = d_S + 2$ intervals if the forward or backward trajectories of neither end point of $X$ under the flow hits a  singularity
\end{itemize}
\end{proposition}

In \cref{section: gap theorem for d-iets via zippered rectangles}, we show that this proposition implies the $d+2$ Gap Theorem (\cref{theorem: d + 2 gap theorem}).

\subsection{Zippered Rectangles Decomposition}
\label{subsection: zippered rectangles decomposition}

Consider a translation surface $S$, a minimal linear flow $\phi$ on $S$, and a transversal $X$ to $\phi$. As per \cref{proposition: number of discontinuities of first return map}, the first return map of $\phi$ to $X$ is an interval exchange map $T$ on $d_T$ subintervals of $X$. It can be shown that the return times of $\phi$ are constant on each of those $d_T$ subintervals. (Check \cite{Yoccoz2007-zz}.) This decomposes the surface $S$ into $d_T$ \textbf{zippered rectangles} whose widths are the lengths $\lambda = (\lambda_1, \lambda_2, \cdots, \lambda_{d_T})$ associated with $T$, and whose heights $h = (h_1, h_2, \cdots, h_{d_T})$ are the return times of $\phi$ to $X$ on each of the $d_T$ subintervals of $X$. The components of the vector $(h, \lambda)$ will be referred to as the \textbf{height-width parameters} of the \textbf{zippered rectangle decomposition} of $S$ with respect to the suspension of $\phi$ over $X$.\footnote{Veech's zippered rectangle construction involve a third parameter: the altitudes of singularities $a$, in addition to the lengths $\lambda$ and heights $h$. For this paper, the altitudes do not serve a significant purpose, and we choose to ignore them.}

\subsection{The Space of Unimodular Lattices}
\label{subsection: the space of unimodular lattices}

Let $X_2$ denote the modular surface $X_2 = \operatorname{SL}_2 \RR / \operatorname{SL}_2 \ZZ$. For any $A \in \operatorname{SL}_2\RR$, the coset $A\cdot\operatorname{SL}_2\RR$, the unimodular lattice $A\cdot\ZZ^2$, and the unit area torus $\RR^2/A\cdot\ZZ^2$ can be identified.

The space $X_2$ inherits an $\operatorname{SL}_2\RR$ action given by left translation/multiplication, along with a projection map $\pi_{\Lambda} : \operatorname{SL}_2\RR \to X_2$ is defined by $A \mapsto A \cdot \Lambda$ for every $\Lambda \in \operatorname{SL}_2\RR$. This gives rise to two important flows on $X_2$: the \textbf{horocycle flow} given by the matrices $h_s = \begin{pmatrix} 1 & 0 \\ -s & 1\end{pmatrix}$, for $s \in \RR$, and the \textbf{geodesic flow} given by the matrices $g_t := \begin{pmatrix} e^{-t} & 0 \\ 0 & e^t \end{pmatrix}$, for $t \in \RR$. In \cref{section: primitive lattice points}, it will be more convenient to work with the scaling matrices  $s_r = \begin{pmatrix} r & 0 \\ 0 & r^{-1} \end{pmatrix}$, with $r > 0$.

The space $X_2$ also comes with an $\operatorname{SL}_2 \RR$-invariant Haar measure which is usually denoted by $\mu_2$. The following is true about the convergence of measures on horocycles under translation by the geoedisic flow.

\begin{theorem}[Equidistribution of Large Horocycles \cite{Eskin1993-ea, Hejhal1996-vo, Weiss2007-dg, Kleinbock2012-oy}]
\label{theorem: equidistribution of large horocycles}
For any finite interval $I = [a,b] \subset \RR$, any bounded continuous function $f : X_2 \to \RR$, and any fixed $\Lambda \in X_2$
\begin{equation*}
\lim_{t \to \infty} \frac{1}{b-a} \int_{\RR} f(g_t h_\alpha \cdot \Lambda)\, d\alpha = \int_{X_2} f\, d\mu_2.
\end{equation*}
\end{theorem}

\section{The Three Gap Theorem and Zippered Rectangles}
\label{section: the three gap theorem and zippered rectangles}

In this section, we relate gaps in orbits of circle rotations to zippered rectangle decompositions of unit area tori, and prove \cref{theorem: Three Gap Theorem}, and \cref{theorem: distribution for TGT}. In \cref{section: primitive lattice points}, we relate zippered rectangle decompositions to primitive lattice points, and develop the machinery needed to indepently derive the gap distribution in \cref{theorem: continuous distribution arising from the three gap theorem}. In \cref{section: gap theorem for d-iets via zippered rectangles}, we show how the same approach can be used to prove \cref{theorem: d + 2 gap theorem}, and \cref{theorem: axiomatic theorem}.

\subsection{Proof of Theorem \ref{theorem: Three Gap Theorem}}

\begin{corollary}
\label{corollary: correspondence}
For any $\alpha \in (0, 1)$, and any integer $N \geq 1$, the pairs
\begin{equation*}
    (\text{gap length}, \text{number of gaps of that length})
\end{equation*}
describing the multiset $\widetilde{\operatorname{Gaps}}_{R_\alpha, N}$ correspond to the height-width parameters $(h, \lambda)$ of the zippered rectangle decomposition of the torus $h_\alpha \cdot \mathbb{T}^2$ over any horizontal segment of length $N$.
\end{corollary}

\begin{proof}
For any $\alpha \in (0, 1)$, the circle rotation $R_\alpha$ can be identified with the return map of the horizontal linear flow $\phi_0^t(x, y) = (x + t, y) \mod 1$ to the unit length vertical $V_{h_\alpha \cdot \mathbb{T}^2}^1 = \{0\} \times [0, 1] \subseteq \mathbb{T}^2$. Denoting the horizontal line segment $\phi_0^{[0, N]}(0, 0)$ of length $N$ by $H_{h_\alpha \cdot \mathbb{T}^2}^{N}$, we can write
\begin{equation*}
    H_{h_\alpha \cdot \mathbb{T}^2}^{N} \cap V_{h_\alpha \cdot \mathbb{T}^2}^1 = \{(0, R_\alpha(0)^n)\}_{n=0}^{N}.
\end{equation*}

The points $\{(0, R_\alpha^n(0))\}_{n=0}^{N}$ order themselves around the closed vertical segment $V_{h_\alpha \cdot \mathbb{T}^2}^1$ the same way the points $\{R_\alpha^n(0)\}_{n=0}^{N}$ order themselves on the interval $[0, 1)$. Moreover, the return time of $\phi_0^t$ to $V_{h_\alpha \cdot \mathbb{T}^2}^1$ is exactly $1$. This implies that $H_{h_\alpha \cdot \mathbb{T}^2}^{N} \cup V_{h_\alpha \cdot \mathbb{T}^2}$ divides $h_\alpha \cdot \mathbb{T}^2$ into $N$ rectangles with unit widths, and whose heights are the gap lengths of $\{R_\alpha^n(0)\}_{n=0}^{N-1}$. The sought for corrspondence immediately follows for $H_{h_\alpha \cdot \mathbb{T}^2}^{N}$, and for any horizontal segment of length $N$ by the translation invariance of $h_\alpha \cdot \mathbb{T}^2$.
\end{proof}

As a consequence of this corollary, we get the following equivalence.

\begin{corollary}
The Three Gap Theorem and the Three Interval Theorem (\cref{proposition: Three Interval Theorem}) are equivalent.
\end{corollary}

This proves the Three Gap Theorem (\cref{theorem: Three Gap Theorem}).

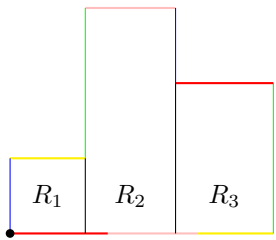
\begin{figure}
\centering
\begin{tikzpicture}[scale = 1]
\draw[red, thick] (0, 0) -- (1.3, 0);
\draw[pink, thick] (1.3, 0) -- (2.5, 0);
\draw[yellow, thick] (2.5, 0) -- (3.5, 0);
\draw[green] (3.5, 0) -- (3.5, 2);
\draw[red, thick] (3.5, 2) -- (2.2, 2);
\draw (2.2, 0) -- (2.2, 2);
\draw[pink, thick] (2.2, 3) -- (1, 3); 
\draw[green] (1, 3) -- (1, 1);
\draw[yellow, thick] (1, 1) -- (0, 1);
\draw (1, 0) -- (1, 1);
\draw[blue] (0, 0) -- (0, 1);
\draw[blue] (2.2, 2) -- (2.2, 3);
\draw node[fill, circle, scale=0.35] at (0,0) {};
\draw node at (0.5, 0.5) {$R_1$} node at (1.6, 0.5) {$R_2$} node at (2.85, 0.5) {$R_3$};
\draw node at (1.9, 2.5) {} node at (3.2, 1.5) {};
\end{tikzpicture}
\caption{A zippered rectangle with colored sides (black excluded) glued together. For $i=1, 2, 3$, the rectangle $R_i$ has width $\lambda_i$ and height $h_i$. This figure tiles the plane.}
\end{figure}

\subsection{The Combinatorial Structure of $\widetilde{\operatorname{Gaps}}_{R_\alpha, N}$, and the zippered rectangle decomposition of $\RR^2/g_{\log N} h_\alpha \cdot \ZZ^2$}
\label{subsection: combinatorial structure of gaps and zippered rectangle decomposition of tori}
\label{subsection: notation}

When $g_{\log N}$ acts on $\RR^2/h_\alpha \cdot \ZZ^2$, it scales it horizontally by $1/N$, and vertically by $N$. In light of \cref{corollary: correspondence}, the combinatorial structure of $\widetilde{\operatorname{Gaps}}_{R_\alpha, N}$ described in \cref{theorem: combinatorial structure theorem}--after scaling all the gaps by $N$--translates to an explicit description of the zippered rectangle decomposition of the unit area tori $\RR^2/g_{\log N} h_\alpha \cdot \ZZ^2$ over any horizontal segment of unit length.

For convenience, we define an operator $\mathfrak{Z}$ that takes a unit area torus $S \in X_2$, and returns the length-height vector $(\lambda, h) \in \RR^4 \sqcup \RR^6$ of its zippered rectangle decomposition over a horizontal line segment of unit length. For the tori $\RR^2/g_{\log N} h_\alpha \cdot \ZZ^2$, we have from \cref{theorem: combinatorial structure theorem} that
\begin{multline*}
\mathfrak{Z}(\RR^2 / g_{\log N} h_{\alpha} \cdot \ZZ^2) =
\begin{cases}
(\frac{q-n_1}{N}, \frac{n_1}{N}, \frac{N}{q}, \frac{N}{q}), & \alpha = \frac{a}{q} \in \mathcal{F}(N) \\
(1 - \frac{q_1}{N}, \frac{q_1}{N} + \frac{q_2}{N} - 1, 1 - \frac{q_2}{N}, \\ N(q_1 \alpha - a_1), N((q_1 - q_2)\alpha - (a_1 - a_2)), N(a_2 - q_2 \alpha)), & \alpha \in \left(\frac{a_1}{q_1}, \frac{a_2}{q_2}\right)
\end{cases}
,
\end{multline*}
where $\mathcal{F}(N)$ are the Farey fractions at level $N$, and $\frac{a_1}{q_1}$ and $\frac{a_2}{q_2}$ are two successive Farey fractions in $\mathcal{F}(N)$. In corollary \ref{corollary: Zippered Rectangle Decompositions of Tori}, we derive the zippered rectangle decompositions for unimodular tori, which makes it possible to retrieve the gap description from theorem \ref{theorem: combinatorial structure theorem} using geometry.

\subsection{Proof of Theorem \ref{theorem: distribution for TGT}, and Setup of Proof of Theorem \ref{theorem: continuous distribution arising from the three gap theorem}}
\label{subsection: gap distribution for circle rotations}

We can now prove \cref{theorem: distribution for TGT}. Moreover, we show in \cref{subsection: proof of continuous distribution theorem} how this gives \cref{theorem: continuous distribution arising from the three gap theorem}.

\begin{proof}[Proof of \cref{theorem: distribution for TGT}]
For convenience, we use the Iverson bracket. For a predicate $P$, the \textbf{Iverson bracket} $[P]$ is
\begin{equation*}
[P] :=
\begin{cases}
1, & \text{ if $P$ is true} \\
0, & \text{ if $P$ is false}
\end{cases}.
\end{equation*}

For any $z \geq 0$, define the cut-off function $d_z : \RR^2 \to \RR$ by \[d_z(x, y) = [y \geq z] x,\]
and define the aggregate cut-off function $f_z : \RR^6 \to \RR$ by
\begin{equation*}
    f_z(\lambda_1, \lambda_2, \lambda_3, h_1, h_2, h_3) = d_z(\lambda_1, h_1) + d_z(\lambda_2, h_2) + d_z(\lambda_3, h_3).
\end{equation*}

It follows from \cref{subsection: notation} that the count of normalized gap lengths greater than or equal to $z$ defined in \cref{subsubsection: the statistics of the gaps of circle rotations} evaluates to
\begin{equation*}
\mathcal{G}_{R_\alpha, N}(z) =
\begin{cases}
[1/q \geq z] q, & \alpha = \frac{a}{q} \in \mathcal{F}(N) \\
N \times f_z(\mathfrak{Z}(g_{\log N} h_\alpha \cdot \ZZ^2)), & \text{ otherwise}
\end{cases}.
\end{equation*}
Integrating we get the average gap distribution defined in \cref{subsubsection: the statistics of the gaps of circle rotations} is
\begin{equation*}
    g_{R_\cdot}^{[a, b]}(z; N) = \frac{1}{b-a} \int_a^b \frac{\mathcal{G}_{R_\alpha, N}(z)}{N}\, d\alpha = \frac{1}{b-a} \int_a^b f_z(\mathfrak{Z}(g_{\log N} h_\alpha \cdot \ZZ^2))\, d\alpha.
\end{equation*}

The uniform probability measure $\nu(\cdot) = \frac{1}{b-a} \int_a^b \cdot\, d\alpha$ is absolutely continuous with respect to the Lebesgue measure on $\RR$. A compactness argument, and the equidistribution of large horocycles (\cref{theorem: equidistribution of large horocycles}) prove \cref{theorem: distribution for TGT}.
\end{proof}

\section{Primitive Lattice Points, the BCZ Map, and Limiting Distributions}
\label{section: primitive lattice points}

If $\frac{a_1}{q_1}, \frac{a_2}{q_2}, \frac{a_3}{q_3} \in \mathcal{F}(N)$ are three consecutive Farey fractions of order $N$, it is well-known that
\begin{equation*}
a_3 = ka_2 - a_1,
\end{equation*}
and 
\begin{equation*}
q_3 = kq_2 - q_1,
\end{equation*}
with $k = \left\lfloor\frac{N+q_1}{q_2}\right\rfloor$. That is, it is possible to generate Farey fractions of order $N$ in increasing order. (They can similarly be generated in decreasing order as well.) In this section, we show that primitive lattice points satisfy the same generative property of Farey fractions, and prove a limiting result that makes it possible to derive \cref{theorem: continuous distribution arising from the three gap theorem}.

\subsection{Primitive Lattice Points}

Given a unimodular lattice $\Lambda \subset \RR^2$, a lattice point $\vec{v} \in \Lambda \setminus \{0\}$ is said to be \textbf{primitive} if the only lattice point in the half-open line segment $[0,1) \cdot \vec{v} = \{r\vec{v} \mid r \in [0,1)\}$ is the origin $\begin{pmatrix} 0 \\ 0 \end{pmatrix}$. We denote by $\Lambda_\text{prim}$ the set of primitive points in $\Lambda$. Note that
\begin{equation*}
\ZZ^2_\text{prim} = \{\begin{pmatrix}q \\ a\end{pmatrix} \in \ZZ^2 \mid \gcd(q,a)=1\}.
\end{equation*}
Also, for any $g \in \SL(2,\RR)$, and any unimodular $\Lambda$,
\begin{equation*}
(g \cdot \Lambda)_\text{prim} = g \cdot \Lambda_\text{prim}.
\end{equation*}

For any $\tau > 0$, denote by $S_\tau$ the vertical strip $(0,\tau]\times \RR$ in the plane. Given two primitive points $\begin{pmatrix}q_1 \\ a_1\end{pmatrix}, \begin{pmatrix}q_2 \\ a_2\end{pmatrix}$ in $\Lambda_\text{prim} \cap S_\tau$, we say that they have \textbf{consecutive slopes} if $\frac{a_1}{q_1} < \frac{a_2}{q_2}$, and there are no other points $\begin{pmatrix}q^\prime \\ a^\prime\end{pmatrix} \in \Lambda_\text{prim} \cap S_\tau$ with $\frac{a_1}{q_1} < \frac{a^\prime}{q^\prime} < \frac{a_2}{q_2}$. Given an interval $I \subseteq \RR$, we write
\begin{equation*}
\mathcal{F}_I(\Lambda, \tau) = \{\begin{pmatrix}q \\ a\end{pmatrix} \in \Lambda_\text{prim} \cap S_\tau \mid \frac{a}{q} \in I\},
\end{equation*}
and
\begin{equation*}
N_{\Lambda, I}(\tau) = \# \mathcal{F}_I(\Lambda, \tau).
\end{equation*}
If $I = \RR$, we write
\begin{equation*}
\mathcal{F}(\Lambda, \tau) = \mathcal{F}_\RR(\Lambda, \tau) = \Lambda_\text{prim} \cap S_\tau.
\end{equation*}
For any integer $N \geq 1$, we can identify $\begin{pmatrix}q \\ a\end{pmatrix} \in \mathcal{F}_{[0,1]}(\ZZ^2, N)$ with the Farey fraction $\frac{a}{q} \in \mathcal{F}(N)$.

\begin{figure}
\label{figure: primitive points example}
\centering
\begin{tikzpicture}
\fill[fill=gray!20] (0,0) -- (3,0.5) -- (3,2.5);
\draw (-0.25,0) -- (4.2,0);
\draw (0,-0.25) -- (0,3.5);
\draw (0,0) -- (3.5,3.5*1/6);
\draw (0,0) -- (3.5,3.5*5/6);
\draw (3,-0.25) -- (3,3.5);
\draw [fill=white] (0,0) circle [radius=0.1em];
\draw [fill=black] (0,1) circle [radius=0.1em];
\draw [fill=white] (0,2) circle [radius=0.1em];
\draw [fill=white] (0,3) circle [radius=0.1em];
\draw [fill=black] (1,0) circle [radius=0.1em];
\draw [fill=black] (1,1) circle [radius=0.1em];
\draw [fill=black] (1,2) circle [radius=0.1em];
\draw [fill=black] (1,3) circle [radius=0.1em];
\draw [fill=white] (2,0) circle [radius=0.1em];
\draw [fill=black] (2,1) circle [radius=0.1em];
\draw [fill=white] (2,2) circle [radius=0.1em];
\draw [fill=black] (2,3) circle [radius=0.1em];
\draw [fill=white] (3,0) circle [radius=0.1em];
\draw [fill=black] (3,1) circle [radius=0.1em];
\draw [fill=black] (3,2) circle [radius=0.1em];
\draw [fill=white] (3,3) circle [radius=0.1em];
\draw [fill=white] (4,0) circle [radius=0.1em];
\draw [fill=black] (4,1) circle [radius=0.1em];
\draw [fill=white] (4,2) circle [radius=0.1em];
\draw [fill=black] (4,3) circle [radius=0.1em];
\draw
node at (3, -0.5) {$\scriptscriptstyle x=\tau$}
node at (3.5, 0.8) {$\scriptscriptstyle y=ax$}
node at (3.5, 3.25) {$\scriptscriptstyle y=bx$};
\end{tikzpicture}
\caption{The set $\mathcal{F}_I(\Lambda, \tau)$ is the collection of primitive points inside the shaded region, with $\Lambda=\ZZ^2$, $I=[1/6,5/6]$, and $\tau=3$. Primitive points are in black, and non-primitive points are in white.}
\end{figure}
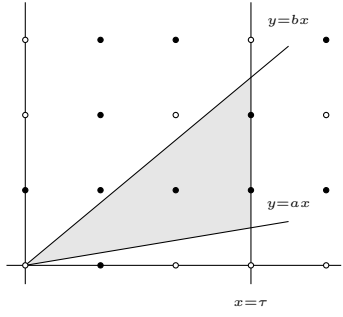

We first state a useful asymptotic result.

\begin{proposition}
\label{proposition: asymptotic growth of N_I}
If $\Lambda \subset \RR^2$ is a unimodular lattice, and $I \subset \RR$ is a finite interval, then
\begin{equation*}
N_{\Lambda, I}(\tau) \sim \frac{3}{\pi^2}|I|\tau^2
\end{equation*}
as $\tau \to \infty$.
\end{proposition}

\begin{proof}
Consider the triangle
\begin{equation*}
A_I := \{\begin{pmatrix}t \\ \alpha t\end{pmatrix} \mid t \in [0,1], \alpha \in I\}.
\end{equation*}
For any $\tau > 0$, the area of the homothetic dilation $\tau A_I$ is $\frac{1}{2}|I|\tau^2$. If $\Lambda = g \cdot \ZZ^2$ is a unimodular lattice, with $g \in \SL(2,\RR)$, then for all $\tau > 0$
\begin{equation*}
N_{\Lambda,I}(\tau) = \#(\Lambda_\text{prim} \cap \tau A_I) = \#(\ZZ_\text{prim}^2 \cap \tau g^{-1} \cdot A_I) = \#(\ZZ_\text{prim}^2 \cap \tau A_I^\prime),
\end{equation*}
with the triangles $\tau A_I^\prime = \tau g^{-1} \cdot A_I$ having the same area as $\tau A_I$.

Now, the triangle $A_I^\prime$ is \emph{starlike} with respect to the origin, and it thus follows from \cite{Nowak1997-gd} that
\begin{equation*}
\#(\ZZ_\text{prim}^2 \cap \tau A_I^\prime) \sim \frac{\operatorname{area}(A_I^\prime)}{\zeta(2)}\tau^2 = \frac{3}{\pi^2}|I|\tau^2,
\end{equation*}
thus proving the claim.
\end{proof}

The following lemma is immediate.\footnote{Note that only matrices of the form $\begin{pmatrix}r & 0 \\ t & r^{-1}\end{pmatrix}$, with $r > 0$ and $t \in \RR$ map the set of lines $x=\tau$, with $\tau > 0$, to itself. That is, only the aforementioned matrices work with our construction $\mathcal{F}(\lambda, \tau)$.} Note that the scaling matrices $s_r$ are related to the geodesic flow matrics $g_t$ by $s_r = g_{-\log r}$. For this section, working with scaling matrices proves to be more convenient.

\begin{lemma}
\label{lemma: scaling and shearing matrices}
For any unimodular lattice $\Lambda \subset \RR^2$, and any $\tau > 0$, the following is true:
\begin{enumerate}
\item If $h_s = \begin{pmatrix} 1 & 0 \\ -s & 1 \end{pmatrix}$, with $s \in \RR$, then
$$h_s \cdot \mathcal{F}_I(\Lambda, \tau) = \mathcal{F}_{I - s}(h_s \cdot \Lambda, \tau).$$
\item If $s_r = \begin{pmatrix} r & 0 \\ 0 & r^{-1} \end{pmatrix}$, with $r > 0$, then $$s_r \cdot \mathcal{F}_I(\Lambda, \tau) = \mathcal{F}_{\frac{1}{r^2}I}(s_r \cdot \Lambda, r\tau).$$
\end{enumerate}
\end{lemma}

\subsubsection{Primitive Points of Lattice with Vertical Vectors}

\begin{lemma}[Lattices with Vertical Vectors]
\label{lemma: lattices with vertical vectors}
Let $\Lambda \subset \RR^2$ be a unimodular lattice with vertical vectors. If $\begin{pmatrix}0 \\ \rho \end{pmatrix}$ is the shortest vertical vector pointing upwards ($\rho > 0$), then there exists a smallest $s \geq 0$ such that the columns of $\begin{pmatrix}\rho^{-1} & 0 \\ s & \rho \end{pmatrix}$ form a positively oriented basis of $\Lambda$. In that case,
\begin{equation*}
\mathcal{F}_I(\Lambda, \tau) = h_{- \rho s} s_\frac{1}{\rho} \cdot \mathcal{F}_\frac{I - \rho s}{\rho^2}(\ZZ^2, \rho \tau).
\end{equation*}
The value $\tau_0 = \frac{1}{\rho}$ is the smallest $\tau > 0$ such that $\mathcal{F}(\Lambda, \tau) \neq \emptyset$, and $\mathcal{F}(\Lambda, \tau) \neq \emptyset$ for all $\tau \geq \tau_0$.
\end{lemma}

\begin{proof}
Let $\begin{pmatrix}r \\ s\end{pmatrix} \in \Lambda$ form with $\begin{pmatrix}\rho \\ 0 \end{pmatrix}$ a positively-oriented basis of $\Lambda$. This implies that $r\rho = \det \begin{pmatrix} r & 0 \\ s & \rho \end{pmatrix} = 1$, which gives $r = \frac{1}{\rho}$. We can assume WLOG that $s \geq 0$ is minimal by adding or subtracting $\begin{pmatrix} 0 \\ \rho \end{pmatrix}$ to $\begin{pmatrix} \frac{1}{\rho} \\ s \end{pmatrix}$ an appropriate number of times. The remainder of the lemma is easy.
\end{proof}

\begin{corollary}
\label{corollary: x components of periodic lattices are periodic}
Let $\Lambda \subset \RR^2$ be a unimodular lattice with vertical vectors. If $\begin{pmatrix}0 \\ \rho\end{pmatrix}$ is the shortest vertical vector in $\Lambda$ pointing upwards ($\rho > 0$), then for any interval $I \subseteq \RR$,
\begin{equation*}
h_{n\rho^2} \cdot \mathcal{F}_I(\Lambda, \tau) = \mathcal{F}_{I - n\rho^2}(\Lambda, \tau).
\end{equation*}
In particular, if $\mathcal{F}_I(\Lambda, \tau) \neq \emptyset$, then the map
\begin{equation*}
h_{n\rho^2} : \mathcal{F}_I(\Lambda, \tau) \to \mathcal{F}_{I - n\rho^2}(\Lambda, \tau)
\end{equation*}
is a bijection that sends any point $\begin{pmatrix}q \\ a\end{pmatrix} \in \mathcal{F}_I(\Lambda, \tau)$ to a point $\begin{pmatrix}p \\ b\end{pmatrix} \in \mathcal{F}_{I - n\rho^2}(\Lambda, \tau)$ with the same $x$-component $p = q$, and slope $\frac{b}{p} = \frac{a}{q} - n\rho^2$.
\end{corollary}

\begin{proof}
It is easy to see for any integer $n \in \ZZ$ that $h_n \cdot \ZZ^2 = \ZZ^2$. From this follows that for any $\tau > 0$ and interval $I \subseteq \RR$
\begin{equation*}
h_n \cdot \mathcal{F}_I(\ZZ^2, \tau) = \mathcal{F}_{I-n}(h_n \cdot \ZZ^2, \tau) = \mathcal{F}_{I-n}(\ZZ^2, \tau).
\end{equation*}

Now, if $\Lambda$ is a unimodular lattice with vertical vectors, and $\begin{pmatrix}0 \\ \rho\end{pmatrix}$ with $\rho > 0$ is the shortest vertical vector in $\Lambda$ pointing upwards, then for any interval $I \subseteq \RR$ and any $\tau > 0$
\begin{equation*}
\mathcal{F}_I(\Lambda, \tau) = h_{- \rho s} s_\frac{1}{\rho} \cdot \mathcal{F}_\frac{I - \rho s}{\rho^2}(\ZZ^2, \rho \tau)
\end{equation*}
from lemma \ref{lemma: scaling and shearing matrices}. A direction computation gives
\begin{eqnarray*}
h_{n\rho^2} (h_{-\rho s} s_{\rho^{-1}}) &=& h_{-\rho s} h_{n\rho^2} s_{\rho^{-1}} \\
&=& h_{-\rho s} s_{\rho^{-1}} h_{n}.
\end{eqnarray*}
From this and \cref{lemma: lattices with vertical vectors}, we get
\begin{eqnarray*}
h_{n\rho^2} \cdot \mathcal{F}_I(\Lambda, \tau) &=& h_{- \rho s} s_\frac{1}{\rho} h_n \cdot \mathcal{F}_\frac{I - \rho s}{\rho^2}(\ZZ^2, \rho \tau) \\
&=& h_{- \rho s} s_\frac{1}{\rho} \cdot \mathcal{F}_{\frac{I - \rho s}{\rho^2}-n}(\ZZ^2, \rho \tau) \\
&=& h_{- \rho s} s_\frac{1}{\rho} \cdot \mathcal{F}_\frac{(I - n \rho^2) - \rho s}{\rho^2}(\ZZ^2, \rho \tau) \\
&=& \mathcal{F}_{I-n\rho^2}(\Lambda, \tau).
\end{eqnarray*}
The rest of the corollary is easy.
\end{proof}

\subsubsection{Primitive Points of Lattices with no Vertical Vectors}

\begin{lemma}[Lattices with no Vertical Vectors]
Let $\Lambda \subset \RR^2$ be a unimodular lattice with no vertical vectors. Then $\mathcal{F}(\Lambda, \tau) \neq \emptyset$ for all $\tau > 0$.
\end{lemma}

\begin{proof}
Consider the rectangle $R_{\tau, h} = [-\tau, \tau] \times [-h, h]$, with $h > \frac{1}{\tau}$. The rectangle $R_{\tau, h}$ has area $(2\tau)(2h) > 2^2$, and so by Minkowski's theorem, it contains a non-zero lattice point, and that point is not vertical as $\Lambda$ has no vertical vectors. That is, there exists a lattice point in $\Lambda$ whose $x$-component is non-zero, and is bounded in absolute value by $\tau$. This implies that $\Lambda \cap S_\tau$, and necessarily $\Lambda_\text{prim} \cap S_\tau = \mathcal{F}(\Lambda, \tau)$, is non-empty.
\end{proof}

\begin{proposition}
Let $\Lambda \subset \RR^2$ be a unimodular lattice with no vertical vectors. For all $\tau > 0$, and any two distinct points $\begin{pmatrix}q \\ a\end{pmatrix}, \begin{pmatrix}q^\prime \\ a^\prime \end{pmatrix} \in \mathcal{F}(\Lambda, \tau)$, $q \neq q^\prime$. In particular, no two distinct points $\begin{pmatrix}q \\ a\end{pmatrix}, \begin{pmatrix}q^\prime \\ a^\prime \end{pmatrix} \in \mathcal{F}(\Lambda, \tau)$ are related by a relationship of the form $h_\alpha \begin{pmatrix}q \\ a\end{pmatrix} = \begin{pmatrix}q^\prime \\ a^\prime \end{pmatrix}$ for any $\alpha \in \RR$.
\end{proposition}

\begin{proof}
If two distinct points $\begin{pmatrix}q \\ a\end{pmatrix}, \begin{pmatrix}q^\prime \\ a^\prime \end{pmatrix} \in \mathcal{F}(\Lambda, \tau)$ have the same $x$-component, then their difference is a vertical non-zero vector in $\Lambda$, which contradicts the assumption that $\Lambda$ has no vertical vectors. The rest follows from the fact that action of $h_\alpha$ preserves the $x$-component.
\end{proof}

\subsection{Primitive Lattice Points, the Farey Triangle, and the BCZ Map}
\label{subsection: Primitive Lattice Points, the Farey Triangle, and the BCZ Map}

We now draw several parallels between Farey fractions, and primitive lattice points in a strip.

\subsubsection{Farey Neighbors}

\begin{lemma}
\label{lemma: Farey Neighbors}
Let $\Lambda$ be a unimodular lattice, $\tau > 0$ with $\mathcal{F}(\Lambda, \tau) \neq \emptyset$, and $\begin{pmatrix}q_1 \\ a_1\end{pmatrix}, \begin{pmatrix}q_2 \\ a_2\end{pmatrix}$ be two points in $\mathcal{F}(\Lambda, \tau)$ with consecutive slopes. The following are true.
\begin{enumerate}[label=(\alph*)]
\item The $x$-components $q_1$ and $q_2$ satisfy $0 < q_1, q_2 \leq \tau$.
\item There are no lattice points from $\Lambda$ in the interior of the triangle bound by the lines $y = \frac{a_1}{q_1}x$, $y = \frac{a_2}{q_2}x$, and $x = \tau$.
\item The $x$-components $q_1$ and $q_2$ satisfy $q_1 + q_2 > \tau$.
\item The vectors $\begin{pmatrix}q_1 \\ a_1\end{pmatrix}$ and $\begin{pmatrix}q_2 \\ a_2\end{pmatrix}$ form a positively-oriented basis of the lattice $\Lambda$.
\item The components of $\begin{pmatrix}q_1 \\ a_1\end{pmatrix}$ and $\begin{pmatrix}q_2 \\ a_2\end{pmatrix}$ satisfy the \textbf{Farey neighbor identity}
\begin{equation*}
a_2 q_1 - a_1 q_2 = 1.
\end{equation*}
Equivalently,
\begin{equation*}
\frac{a_2}{q_2} = \frac{a_1}{q_1} + \frac{1}{q_1q_2}.
\end{equation*}
\end{enumerate}
\end{lemma}

\begin{proof}
\begin{enumerate}[label=(\alph*)]
\item The points $\begin{pmatrix}q_1 \\ a_1\end{pmatrix}$ and $\begin{pmatrix}q_2 \\ a_2\end{pmatrix}$ belong to the strip $S_\tau$, and so their $x$-components belong to $(0,\tau]$.
\item If there exists a lattice point $\vec{u}$ in the interior of the aforementioned triangle, then the triangle contains a primitive lattice point $\vec{u}_0 = r_0 \vec{u}$, with $r_0$ be the smallest $r > 0$ such that $r\vec{u} \in \Lambda$. The point $\vec{u}_0$ is inside the strip $S_\tau$. The slope of $\vec{u}_0$ is included in $\left(\frac{a_1}{q_1}, \frac{a_2}{q_2}\right)$, which contradicts $\begin{pmatrix}q_1 \\ a_1\end{pmatrix}$ and $\begin{pmatrix}q_2 \\ a_2\end{pmatrix}$ having consecutive slopes.
\item The sum $\vec{u} = \begin{pmatrix}q_1 \\ a_1\end{pmatrix} + \begin{pmatrix}q_2 \\ a_2\end{pmatrix}$ is a lattice point with slope included in $\left(\frac{a_1}{q_1}, \frac{a_2}{q_2}\right)$. The $x$-component of $\vec{u}$ is $q_1 + q_2$, and if $q_1 + q_2 \leq \tau$, then $\vec{u}$ belongs to the strip $S_\tau$, which contradicts the statement in this lemma. This implies that $q_1 + q_2 > \tau$.
\item Assume that $\vec{v}_1 = \begin{pmatrix}q_1 \\ a_1\end{pmatrix}$ and $\vec{v}_2 = \begin{pmatrix}q_2 \\ a_2\end{pmatrix}$ do \emph{not} form a basis of $\Lambda$. By a well-known property of (unimodular) lattices, then there must be a \emph{non-zero} lattice point $\vec{u} = \begin{pmatrix}u_1 \\ u_2\end{pmatrix} = \alpha_1 \vec{v}_1 + \alpha_2 \vec{v}_2$, with $\alpha_1, \alpha_2 \in [0,1)$. Note that the slope of $\vec{u}$ is included in $\left(\frac{a_1}{q_1}, \frac{a_2}{q_1}\right)$, and so must have $u_1 > \tau$ by the first statement in this lemma. Now consider the vector $\vec{w} = \vec{v}_1 + \vec{v}_2 - \vec{u}$:
  \begin{enumerate}
  \item It is a lattice point.
  \item Its $x$-component $q_1 + q_2 - u_1$ satisfies
   \begin{equation*}
   0 < q_1 + q_2 - u_1 < \tau + \tau - \tau = \tau.
   \end{equation*}
   That is, $\vec{u}$ belongs to the strip $S_\tau$.
  \item It is on the form
   \begin{equation*}
   \vec{w} = (1-\alpha_1)\vec{v}_1 + (1-\alpha_2)\vec{v}_2,
   \end{equation*}
   with $\alpha_1, \alpha_2 \in (0,1]$. That is, the slope of $\vec{w}$ is included in $\left(\frac{a_1}{q_1}, \frac{a_2}{q_1}\right)$.
  \end{enumerate}
The point $\vec{w}$, and by extension $\vec{u}$, thus cannot exist by the first statement in this lemma. Then $\vec{v}_1$ and $\vec{v}_2$ form a basis of $\Lambda$. They are positively-oriented as the slope of $\vec{v}_2$ is larger than that of $\vec{v}_1$.
\item Since $\begin{pmatrix}q_1 \\ a_1\end{pmatrix}$ and $\begin{pmatrix}q_2 \\ a_2\end{pmatrix}$ form a positively-oriented basis of the unimodular lattice $\Lambda$, they satisfy
\begin{equation*}
1 = \det\begin{pmatrix}q_1 & q_2 \\ a_1 & a_2\end{pmatrix} = a_2 q_1 - a_1 q_2 = 1.
\end{equation*}
\end{enumerate}
\end{proof}

\begin{corollary}
Let $\Lambda \subset \RR^2$ be a unimodular lattice, $\tau > 0$ with $\mathcal{F}(\Lambda, \tau) \neq \emptyset$, and $\begin{pmatrix}q \\ a\end{pmatrix} \in \Lambda_\text{prim} \cap S_\tau$ be a primitive vector. There exists a unique $\begin{pmatrix}q^\prime \\ a^\prime\end{pmatrix} \in \Lambda_\text{prim} \cap S_\tau$ such that $\begin{pmatrix}q \\ a\end{pmatrix}$ and $\begin{pmatrix}q^\prime \\ a^\prime\end{pmatrix}$ have consecutive slopes, and a unique $\begin{pmatrix}q^{\prime\prime} \\ a^{\prime\prime}\end{pmatrix}$ such that $\begin{pmatrix}q^{\prime\prime} \\ a^{\prime\prime}\end{pmatrix}$ and $\begin{pmatrix}q \\ a\end{pmatrix}$ have consecutive slopes.
\end{corollary}

\begin{proof}
If $\Lambda$ has a vertical vector, then the corollary is obviously true, and so we assume WLOG that $\Lambda$ has \emph{no} vertical vectors.

It suffices to prove that $\Lambda \cap S_\tau$ contains points whose slopes are strictly smaller and bigger than $\frac{a}{q}$, and for that we use Minkowski's theorem. Consider the rectangle $R_{\frac{q}{2},h} = [-q/2,q/2] \times [-h,h]$, with $h > \frac{2}{q}$. This rectangle has area bigger than $2^2$, and so by Minkowski's theorem contains a non zero lattice point $\vec{v}_0$. Note that $\vec{v}_0$ is not a scalar multiple of $\begin{pmatrix} q \\ a \end{pmatrix}$, and is not vertical. Now, the two vectors $\begin{pmatrix} q \\ a \end{pmatrix} \pm \vec{v}_0$ are in $\Lambda \cap S_\tau$, have finite slopes that are stricly bigger and smaller than that of $\frac{a}{q}$. This implies that $\Lambda \cap S_\tau$, and necessarily $\Lambda_\text{prim} \cap S_\tau$, contains points with slopes strictly bigger and smaller than $\frac{a}{q}$, which proves the claim.
\end{proof}

\subsubsection{Primitive Lattice Points, and the Farey Triangle $\mathscr{T}$}

\begin{corollary}
\label{corollary: primitive lattice points and the Farey triangle}
Let $\Lambda$ be a unimodular lattice, $\tau > 0$ with $\mathcal{F}(\Lambda, \tau) \neq \emptyset$, and $\begin{pmatrix}q_1 \\ a_1\end{pmatrix}, \begin{pmatrix}q_2 \\ a_2\end{pmatrix} \in \mathcal{F}(\Lambda, \tau)$ be two primitive vectors with consecutive slopes.
\begin{enumerate}[label=(\alph*)]
\item For $\alpha \in \left[\frac{a_1}{q_1}, \frac{a_2}{q_2}\right]$, the lattices $h_\alpha \cdot \Lambda$ have horizontal lattice points with lengths not exceeding $\tau$ only at $\alpha = \frac{a_1}{q_1}, \frac{a_2}{q_2}$.
\item The columns of the matrix $\begin{pmatrix}\frac{q_1}{\tau} & \frac{q_2}{\tau} \\ 0 & \left(\frac{q_1}{\tau}\right)^{-1}\end{pmatrix}$ form a positively-oriented basis of $s_\frac{1}{\tau} h_\frac{a_1}{q_1} \cdot \Lambda = h_{\frac{a_1}{q_1}\tau^2}(s_\frac{1}{\tau} \cdot \Lambda)$, and are two primitive points with consecutive slopes in $\mathcal{F}\left(s_\frac{1}{\tau} h_\frac{a_1}{q_1} \cdot \Lambda, 1\right) = \left(s_\frac{1}{\tau} h_\frac{a_1}{q_1} \cdot \Lambda\right)_\text{prim} \cap S_1$, with $\begin{pmatrix}\frac{q_1}{\tau} \\ 0 \end{pmatrix}$ being the shortest horizontal vector of $s_\frac{1}{\tau} h_\frac{a_1}{q_1} \cdot \Lambda$ with a positive $x$-component. \item The lattice $s_\frac{1}{\tau} h_\frac{a_1}{q_1} \cdot \Lambda = h_{\frac{a_1}{q_1}\tau^2}(s_\frac{1}{\tau} \cdot \Lambda)$ can be identified with the point $\left(\frac{q_1}{\tau}, \frac{q_2}{\tau}\right)$ in the \textbf{Farey triangle} \cite{Boca2001-he}
\begin{equation*}
\mathscr{T} = \{ (x,y) \in \RR^2 \mid 0 < x, y, \leq 1, x + y > 1 \}.
\end{equation*}
\item If $\begin{pmatrix}q_1 \\ a_1\end{pmatrix}, \begin{pmatrix}q_2 \\ a_2\end{pmatrix}, \begin{pmatrix}q_3 \\ a_3\end{pmatrix} \in \mathcal{F}(\Lambda, \tau)$ have consecutive slopes, then
\begin{equation*}
T\left(\frac{q_1}{\tau}, \frac{q_2}{\tau}\right) = \left(\frac{q_2}{\tau}, \frac{q_3}{\tau}\right),
\end{equation*}
where $T : \mathscr{T} \to \mathscr{T}$ is the \textbf{BCZ map} \cite{Boca2001-he}
\begin{equation*}
T(x, y) = \left(y,\left\lfloor \frac{1+x}{y} \right\rfloor y - x \right).
\end{equation*}
\end{enumerate}
\end{corollary}

\begin{proof}
\begin{enumerate}[label=(\alph*)]
\item This follows from the second statement in \cref{lemma: Farey Neighbors}, and the fact that the triangular region bound by $y=\frac{a_1}{q_1}x$, $y=\frac{a_2}{q_2}x$, and $x=\tau$ is mapped by $h_\alpha$ to the triangular region bound by $y=\left(\frac{a_1}{q_1} - \alpha\right)x$, $y=\left(\frac{a_2}{q_2} - \alpha\right)x$, and $x=\tau$, which will continue to have no lattice points from $h_\alpha \cdot \Lambda$ in its interior.
\item By the third statement in \cref{lemma: Farey Neighbors}, the columns of the matrix $\begin{pmatrix}q_1 & q_2 \\ a_1 & a_2\end{pmatrix}$ form a positively-oriented basis of $\Lambda$. This gives a positively-oriented basis of $s_\frac{1}{\tau}h_\frac{a_1}{q_1} \cdot \Lambda$
\begin{eqnarray*}
s_\frac{1}{\tau} h_\frac{a_1}{q_1} \begin{pmatrix}q_1 & q_2 \\ a_1 & a_2\end{pmatrix} &=& \begin{pmatrix}\frac{1}{\tau} & 0 \\ 0 & \tau\end{pmatrix} \begin{pmatrix}1 & 0 \\ -\frac{a_1}{q_1} & 1\end{pmatrix} \begin{pmatrix}q_1 & q_2 \\ a_1 & a_2\end{pmatrix} \\
&=& \begin{pmatrix}\frac{1}{\tau} & 0 \\ 0 & \tau\end{pmatrix} \begin{pmatrix}q_1 & q_2 \\ 0 & \frac{1}{q_i}\end{pmatrix} \\
&=& \begin{pmatrix}\frac{q_1}{\tau} & \frac{q_2}{\tau} \\ 0 & \left(\frac{q_1}{\tau}\right)^{-1}\end{pmatrix}
\end{eqnarray*}
where $a_2q_1 - a_1q_2 = 1$ from \cref{lemma: Farey Neighbors} has been used.
\item Given a point $(r, s)$ in the Farey triangle $\mathscr{T}$, we have $\begin{pmatrix} r & s \\ 0 & r^{-1}\end{pmatrix} \in \SL(2, \RR)$ with determinant $1$, and so its columns form a positively-oriented basis of the unimodular lattice $\Lambda_{r,s} = \begin{pmatrix} r & s \\ 0 & r^{-1}\end{pmatrix} \cdot \ZZ^2$. Since $\vec{v}_1 = \begin{pmatrix}r \\ 0\end{pmatrix}$ and $\vec{v}_2 = \begin{pmatrix}s \\ r^{-1}\end{pmatrix}$ form a basis of $\Lambda_{r,s}$, then the only lattice point in $W_{r,s} = [0,1)\cdot\vec{v}_1 + [0,1)\cdot\vec{v}_2$ is the origin $\vec{0}$. This implies that $\vec{v}_1$ and $\vec{v}_2$ are primitive lattice points of $\Lambda_{r,s}$, and it remains to show that they have consecutive slopes. Consider the wedge $\tilde{W}_{r,s} = \cup_{n,m \in \mathbb{N}}W_{a,b} + n\vec{v}_1 + m\vec{v}_2$. The two lines $y=0$ and $y=\frac{1}{rs}x$ parallel to $\vec{v}_1$ and $\vec{v}_2$ bound $\tilde{W}_{r,s}$, and the only lattice points in $\tilde{W}_{r,s}\cap\Lambda_{a,b}$ lying strictly between $y=0$ and $y=\frac{1}{rs}x$ are of the form $n\vec{v}_1 + m\vec{v}_2$ with $n,m>0$. Since $r+s>1$, the $x$-components of all the points $n\vec{v}_1 + m\vec{v}_2$ with $n,m>0$ are greater than $1$, and hence lie on the left of the line $x=1$. So, the  interior of the triangle bound by $y=0$, $y=\frac{1}{rs}x$, and $x=1$ contains no lattice points, and the columns of $\begin{pmatrix} r & s \\ 0 & r^{-1}\end{pmatrix}$ correspond to primitive lattice points with consecutive slopes in $(\Lambda_{r,s})_\text{prim} \cap S_1$.

It remains to show that given two points $(r, s), (w, z) \in \mathscr{T}$ in the Farey triangle, we the unimodular lattices $\Lambda_{r,s} = \begin{pmatrix}r & s \\ 0 & r^{-1}\end{pmatrix}\ZZ^2$ and $\Lambda_{w,z} = \begin{pmatrix}w & z \\ 0 & w^{-1}\end{pmatrix}\ZZ^2$ are equal only if $(r, s) = (w, z)$. The two lattices $\Lambda_{(r,s)}, \Lambda_{(w,z)}$ are equal if and only if there exists a matrix $\begin{pmatrix}l & m \\ n & k\end{pmatrix} \in \SL(2, \ZZ)$ such that
\begin{equation*}
\begin{pmatrix}r & s\\ 0 & r^{-1}\end{pmatrix}^{-1} \begin{pmatrix}w & z \\ 0 & w^{-1}\end{pmatrix} = \begin{pmatrix}r^{-1}w & r^{-1}z - sw^{-1} \\ 0 & rw^{-1}\end{pmatrix} = \begin{pmatrix}l & m \\ n & k\end{pmatrix}.
\end{equation*}
From this follows that $w=rl$, $r=wk$, so $r=rlk$, from which $l = k = 1$, and so $r = w$. From $z-s=rm$ follows the following
\begin{eqnarray*}
rm &=& z - s \\
&=& (r+z) - (r+s) \\
&=& (w+z) - (a+s) \\
&<& (r+1) - (1) \\
&=& a,
\end{eqnarray*}
and $0 \leq rm < r$ implies that $m=0$. Obviously, $n=0$, and we are done.

\item The points $\left(\frac{q_1}{\tau}, \frac{q_2}{\tau}\right)$ and $\left(\frac{q_2}{\tau}, \frac{q_3}{\tau}\right)$ in the Farey triangle $\mathscr{T}$ can be identified with the lattices $\Lambda_{\frac{q_1}{\tau}, \frac{q_2}{\tau}} = h_{\frac{a_1}{q_1}\tau^2}(s_\frac{1}{\tau} \cdot \Lambda)$ and $\Lambda_{\frac{q_2}{\tau}, \frac{q_3}{\tau}} = h_{\frac{a_2}{q_2}\tau^2}(s_\frac{1}{\tau} \cdot \Lambda)$. Note for $\alpha \in \left[\frac{a_1}{q_1},\frac{a_2}{q_2}\right]$, the lattices $h_{\alpha\tau^2}(s_\frac{1}{\tau}\Lambda)$ have horizontal lattices points with lengths not exceeding $1$ only at $\alpha = \frac{a_1}{q_1},\frac{a_2}{q_2}$. From this, and \cite[lemma 2.2]{Athreya2013-ql}, follows that $T\left(\frac{q_1}{\tau}, \frac{q_2}{\tau}\right) = \left(\frac{q_2}{\tau}, \frac{q_3}{\tau}\right)$.
\end{enumerate}
\end{proof}

\subsubsection{Generating Primitive Lattice Points}

\begin{theorem}
\label{theorem: generating primitive lattice points}
Let $\Lambda$ be a unimodular lattice, $\tau >0$ with $\mathcal{F}(\Lambda, \tau)$, and $\begin{pmatrix}q_1 \\ a_1\end{pmatrix}, \begin{pmatrix}q_2 \\ a_2\end{pmatrix}, \begin{pmatrix}q_3 \\ a_3\end{pmatrix} \in \mathcal{F}(\Lambda, \tau)$ be three primitive points with consecutive slopes. Then
\begin{equation*}
\begin{pmatrix}q_3 \\ a_3\end{pmatrix} = k \begin{pmatrix}q_2 \\ a_2\end{pmatrix} - \begin{pmatrix}q_1 \\ a_1\end{pmatrix},
\end{equation*}
where $k = \left\lfloor \frac{\tau + q_1}{q_2} \right\rfloor$, and
\begin{equation*}
\begin{pmatrix}q_1 \\ a_1\end{pmatrix} = k^\prime \begin{pmatrix}q_2 \\ a_2\end{pmatrix} - \begin{pmatrix}q_3 \\ a_3\end{pmatrix},
\end{equation*}
where $k^\prime = \left\lfloor \frac{\tau + q_3}{q_2} \right\rfloor$.
In particular
\begin{equation*}
\frac{a_2}{q_2} = \frac{a_1 + a_3}{q_1 + q_3}.
\end{equation*}
\end{theorem}

\begin{proof}
That $q_3 = k q_2 - q_1$ follows from $T\left(\frac{q_1}{\tau}, \frac{q_2}{\tau}\right) = \left(\frac{q_2}{\tau}, \frac{q_3}{\tau}\right)$. To prove that $a_3 = k a_2 - a_1$, we use $q_3 = k q_2 - q_1$, along with two applications of the Farey neighbor identity from \cref{lemma: Farey Neighbors}:
\begin{eqnarray*}
a_3 &=& q_3 \times \frac{a_3}{q_3} \\
&=& q_3 \left(\frac{a_2}{q_2} + \frac{1}{q_2 q_3}\right) \\
&=& (k q_2 - q_1) \frac{a_2}{q_2} + \frac{1}{q_2} \\
&=& k a_2 + (1 - q_1a_2)\frac{1}{q_2} \\
&=& k a_2 + (-a_1q_2) \frac{1}{q_2} \\
&=& k a_2 - a_1
\end{eqnarray*}
The rest follows similarly from the BCZ map $T : \mathscr{T} \to \mathscr{T}$ having an inverse $T^{-1} : \mathscr{T} \to \mathscr{T}$ given by \cite{Boca2001-he}
\begin{equation*}
T(x,y) = \left(\left\lfloor\frac{1+y}{x}\right\rfloor x -y, x\right).
\end{equation*}
\end{proof}

\subsection{Best Approximations by Primitive Lattice Points}
\label{subsection: Best Approximations by Primitive Lattice Points}

Note that by \cref{lemma: Farey Neighbors}, if $\Lambda$ is a unimodular lattice, $\tau > 0$ with $\mathcal{F}(\Lambda, \tau) \neq \emptyset$, and $\begin{pmatrix}q_1 \\ a_1\end{pmatrix}, \begin{pmatrix}q_2 \\ a_2\end{pmatrix} \in \mathcal{F}(\Lambda, \tau)$ are two primitive points with consecutive slopes, then
\begin{equation*}
\frac{a_2}{q_2} - \frac{a_1}{q_1} = \frac{1}{q_1q_2} \geq \frac{1}{\tau^2}.
\end{equation*}
That is, the slopes of the elements of $\mathcal{F}(\Lambda, \tau)$ do \emph{not} accumulate. This makes it possible to find best upper and lower approximates of any real number by points of $\mathcal{F}(\Lambda, \tau)$.

\begin{definition}
Let $\Lambda \subset \RR^2$ be a unimodular lattice, and $\tau > 0$ with $\mathcal{F}(\Lambda, \tau)$. We define the best upper and lower approximation maps $\approxplus, \approxminus : \RR \to \mathcal{F}(\Lambda, \tau)$ as follows: for every $\alpha \in \RR$, $\approxplus_{\Lambda, \tau}(\alpha)$ is the element $\begin{pmatrix}q \\ a\end{pmatrix} \in \mathcal{F}(\Lambda, \tau)$ with the smallest slope $\frac{a}{q} \geq \alpha$, and $\approxminus_{\Lambda, \tau}(\alpha)$ is the element $\begin{pmatrix}q^\prime \\ a^\prime\end{pmatrix} \in \mathcal{F}(\Lambda, \tau)$ with the largest slope $\frac{a^\prime}{q^\prime} \leq \alpha$. For convenience, we write
\begin{equation*}
\slope_{\Lambda, \tau}^\pm(\alpha) = \slope\left(\approxpm_{\Lambda, \tau}(\alpha)\right)
\end{equation*}
for the slopes of best approximates by primitive lattice points.
\end{definition}

The following lemma is immediate.

\begin{lemma}
\label{lemma: scaling and shearing approximations}
Let $\Lambda$ be a unimodular lattice, and $\tau > 0$ with $\mathcal{F}(\Lambda, \tau) \neq \emptyset$. The following are true.
\begin{enumerate}[label=(\alph*)]
\item For any $s \in \RR$,
\begin{equation*}
h_s \approxpm_{\Lambda, \tau}(\alpha) = \approxpm_{h_s  \cdot \Lambda, \tau}(\alpha-s).
\end{equation*}
\item For any $r > 0$,
\begin{equation*}
s_r \approxpm_{\Lambda, \tau}(\alpha) = \approxpm_{s_r \cdot \Lambda, r\tau}(\frac{1}{r^2}\alpha).
\end{equation*}
\end{enumerate}
\end{lemma}

\begin{lemma}
\label{lemma: limit of slope of best approximations}
The slopes of any unimodular lattice $\Lambda \subset \RR^2$ are dense in the projective real line $\mathbb{RP}^1$. In particular, for any $\alpha \in \RR$
\begin{equation*}
\lim_{\tau \to \infty} \slope^\pm_{\Lambda, \tau}(\alpha) = \alpha.
\end{equation*}
\end{lemma}

\begin{proof}
The slopes of the points in $\ZZ^2$ are the rational points in $\mathbb{RP}^1$, and hence are dense. Given any unimodular lattice $\Lambda = g \cdot \ZZ^2$, with $g = \begin{pmatrix}a & b \\ c & d\end{pmatrix}$, the slopes of $\Lambda$ are the images of the rational points in $\mathbb{RP}^1$ under the rational maps $g(s) = \frac{c+ds}{a+bs}$, which is an automorhpism of the projective line, and hence are dense. The remainder of the lemma is obvious.
\end{proof}

We also get the following useful identity. Note that when we write $\sum_{i=0}^{N_{\Lambda,I}(\tau)-1} \frac{1}{q_i q_{i+1}}$, the last term involves $q_{N_{\Lambda,I}(\tau)}$, which is \emph{not} the $x$-component of any of the elements of $\mathcal{F}_I(\Lambda,\tau)$. We write this with the understanding that $\vec{v}_1 = \begin{pmatrix}q_{N_{\Lambda,I}(\tau) - 1} \\ a_{N_{\Lambda,I}(\tau) - 1}\end{pmatrix}$ and $\vec{v}_2 = \begin{pmatrix}q_{N_{\Lambda,I}(\tau)} \\ a_{N_{\Lambda,I}(\tau)}\end{pmatrix}$ have consecutive slopes in $\mathcal{F}(\Lambda,\tau)$. That is, $\vec{v}_1$ is the last element in $\mathcal{F}(\Lambda,\tau)$ with slope in $I$, and $\vec{v}_2$ is the following element, necessarily with a slope not in $I$. This should not cause any confusion, and for brevity we do not point it out whenever the aforementioned sums are involved.

\begin{lemma}
\label{lemma: sum of product of reciprocals}
Let $\Lambda$ be a unimodular lattice, $I = [a,b] \subset \RR$ a finite interval, and $\tau > 0$ with $\mathcal{F}_I(\Lambda, \tau) \neq \emptyset$. Writing $\mathcal{F}_I(\Lambda, \tau) = \left\{\begin{pmatrix}q_i \\ a_i\end{pmatrix}\right\}_{i=0}^{N_{\Lambda,I}(\tau) - 1}$, with the elements in increasing slope order, we have
\begin{equation*}
\sum_{i=0}^{N_{\Lambda,I}(\tau)-1} \frac{1}{q_i q_{i+1}} = \slope^+_{\Lambda,\tau}(b) - \slope^+_{\Lambda,\tau}(a).
\end{equation*}
In particular,
\begin{equation*}
\lim_{\tau \to \infty} \sum_{i=0}^{N_{\Lambda,I}(\tau)-1} \frac{1}{q_i q_{i+1}} = b - a = |I|.
\end{equation*}
\end{lemma}

\begin{proof}
By the Farey neighbor identity from \cref{lemma: Farey Neighbors}, and mathematical induction, we get that
\begin{equation*}
\slope^+_{\Lambda,\tau}(a) + \sum_{i=0}^{N_{\Lambda,I}(\tau)-1} \frac{1}{q_i q_{i+1}} = \slope^+_{\Lambda,\tau}(b).
\end{equation*}
The limit follows from \cref{lemma: limit of slope of best approximations}.
\end{proof}

\subsection{Limiting Distributions of Primitive Points}
\label{subsection: Limiting Distributions of Primitive Points}

In \cite{Athreya2013-ql}, the following theorem on the Farey triangle as a cross section to the horocycle flow on $X_2$, along with a number of limiting Farey fraction distribution results, were proved. In the remainder of this section, we use the previous results from this section to extend the aforementioned limiting Farey distribution results to the primitive points in all lattices.

\begin{theorem}[\cite{Athreya2013-ql}]
\label{theorem: Farey triangle as cross section}
Let $\mathscr{T}$ be the Farey triangle, $dm = 2dxdy$ be twice the Lebesgue measure on $\mathscr{T}$ in the plane, and $T : \mathscr{T} \to \mathscr{T}$ be the BCZ map. Let $X_2$ be the space of unimodular lattices, $\mu_2$ the Haar measure on $X_2$, and $h_\cdot$ the horocycle flow on $X_2$. The following are true.
\begin{enumerate}[label=(\alph*)]
\item For any $(x, y) \in \mathscr{T}$, denote $p_{x,y} = \begin{pmatrix}x & y \\ 0 & x^{-1}\end{pmatrix}$. The Farey triangle $\mathscr{T}$ can be bijectively identified with the set $\Omega = \{p_{x,y}\cdot\ZZ^2 \mid (x,y) \in \mathscr{T}\} \subset X_2$ via the map
\begin{equation*}
\Lambda_{\cdot, \cdot} : \mathscr{T} \to \Omega
\end{equation*}
defined for all $(x,y) \in \mathscr{T}$ by
\begin{equation*}
\Lambda_{x,y} = p_{x,y}\cdot\ZZ^2.
\end{equation*}
\item The triple $(\mathscr{T}, m, T)$, with $\mathscr{T}$ identified with $\Omega$, is a cross section to $(X_2, \mu_2, h_\cdot)$, with the roof function $R : \mathscr{T} \to \RR_+$ defined by
\begin{equation*}
R(x,y) = \frac{1}{xy}
\end{equation*}
for all $(x,y) \in \mathscr{T}$.
\end{enumerate}
\end{theorem}

\subsubsection{Distributions Related to Primitive Lattice Points}

For any unimodular lattice $\Lambda$, $\tau > 0$ with $\mathcal{F}_I(\Lambda,\tau) \neq \emptyset$, we write $\mathcal{F}_I(\Lambda, \tau) = \left\{\begin{pmatrix}q_i \\ a_i\end{pmatrix}\right\}_{i=0}^{N_{\Lambda,I}(\tau) - 1}$, with the elements in increasing slope order. We define the measures
\begin{equation*}
\rho_{\Lambda, I, \tau} = \frac{1}{N_{\Lambda, I}(\tau)} \sum_{i=0}^{N_{\Lambda, I}(\tau) - 1} \delta_{T^i\left(\frac{q_0}{\tau}, \frac{q_1}{\tau}\right)}
\end{equation*}
on the Farey triangle $\mathscr{T}$ in the plane. The following theorem shows that those measures converge to $m$ on $\mathscr{T}$.

\begin{theorem}
\label{theorem: equidistribution of x-component measures}
For any unimodular lattice $\Lambda$, finite interval $I = [a,b] \subset \RR$, and $\tau > 0$ with $\mathcal{F}_I(\Lambda, \tau) \neq \emptyset$, the measures $\rho_{\Lambda, I, \tau}$ weak-$\ast$ convergence
\begin{equation*}
\rho_{\Lambda, I, \tau} \rightharpoonup m.
\end{equation*}
\end{theorem}

\begin{proof}
Given a continuous, bounded function $f : X_2 \to \RR$, we have
\begin{eqnarray*}
\frac{1}{b-a} \int_a^b f(s_\frac{1}{\tau} h_\alpha \cdot \Lambda)\, d\alpha &=& \frac{1}{b-a} \int_a^{\slope_{\Lambda, \tau}^+(a)} + \int_{\slope_{\Lambda, \tau}^+(a)}^{\slope_{\Lambda, \tau}^+(b)} - \int_b^{\slope_{\Lambda, \tau}^+(b)} f(s_\frac{1}{\tau} h_\alpha \cdot \Lambda)\, d\alpha \\
&=& \frac{1}{b-a} \int_{\slope_{\Lambda, \tau}^+(a)}^{\slope_{\Lambda, \tau}^+(b)} f(s_\frac{1}{\tau} h_\alpha \cdot \Lambda)\, d\alpha + O(\slope_{\Lambda, \tau}^+(a)-a) + O(\slope_{\Lambda, \tau}^+(b)-b) \\
&=& \frac{1}{b-a} \int_{\slope_{\Lambda, \tau}^+(a)}^{\slope_{\Lambda, \tau}^+(b)} f(s_\frac{1}{\tau} h_\alpha \cdot \Lambda)\, d\alpha + o(1).
\end{eqnarray*}
From \cref{lemma: limit of slope of best approximations} we thus get
\begin{equation*}
\lim_{\tau \to \infty} \frac{1}{\slope_{\Lambda, \tau}^+(b) - \slope_{\Lambda, \tau}^+(a)} \int_{\slope_{\Lambda, \tau}^+(a)}^{\slope_{\Lambda, \tau}^+(b)} f(s_\frac{1}{\tau} h_\alpha \Lambda)\, d\alpha = \lim_{\tau \to \infty} \frac{1}{b-a} \int_a^b f(s_\frac{1}{\tau} h_\alpha \Lambda)\, d\alpha.
\end{equation*}
From this and theorem \ref{theorem: equidistribution of large horocycles} follows the weak-$\ast$ convergence
\begin{equation*}
(s_\frac{1}{\tau})_\ast(\pi_\Lambda)_\ast(h_\cdot)_\ast\operatorname{Unif}_{[{\slope_{\Lambda, \tau}^+(a)}, {\slope_{\Lambda, \tau}^+(b)}]} \rightharpoonup \mu_2,
\end{equation*}
where $\operatorname{Unif}_{[{\slope_{\Lambda, \tau}^+(a)}, {\slope_{\Lambda, \tau}^+(b)}]}$ is the uniform measure on $[{\slope_{\Lambda, \tau}^+(a)}, {\slope_{\Lambda, \tau}^+(b)}]$.

By \cref{theorem: Farey triangle as cross section}, $X_2$ is a suspension over the triangle $\mathscr{T}$ with roof function $R$, and so
\begin{equation*}
\rho_{\Lambda, I, \tau}^R := \rho_{\Lambda, I, \tau}dt = (s_\frac{1}{\tau})_\ast(\pi_\Lambda)_\ast(h_\cdot)_\ast\operatorname{Unif}_{[{\slope_{\Lambda, \tau}^+(a)}, {\slope_{\Lambda, \tau}^+(b)}]},
\end{equation*}
from which
\begin{equation*}
\rho_{\Lambda, I, \tau}^R \rightharpoonup \mu_2
\end{equation*}
by the weak-$\ast$ convergence we have just proved. Writing $\pi : X_2 \to \mathscr{T}$ for the projection map from $X_2$ to the cross section $\mathscr{T}$, $\pi$ is continuous except on a set of measure zero with respect to $\rho_{\Lambda, I, \tau}^R$ and $\mu_2$. From this
\begin{equation*}
\rho_{\Lambda, I, \tau} = \frac{1}{R} \pi_\ast \rho_{\Lambda, I, \tau}^R \rightharpoonup \frac{1}{R} \pi_\ast \mu_2 = m,
\end{equation*}
and we are done.
\end{proof}

The next lemma generalizes \cite[lemma 5.2]{Athreya2013-ql}, and makes it possible to show convergence of integrals with respect to the measures $\rho_{\Lambda, I, \tau}$ for a large family of functions.

\begin{lemma}
\label{lemma: family of convergent integrals}
Let $\Lambda$ be a unimodular lattice, $I \subset \RR$ a finite interval, and $\tau_{\Lambda,I}$ be the infimum of the values $\tau > 0$ such that $\mathcal{F}_I(\Lambda, \tau) \neq \emptyset$. Denoting the measure $m$ on $\mathscr{T}$ by $\rho_{\Lambda, I, \infty}$, if a measureable function $f : \mathscr{T} \to \RR$ satisfies
\begin{equation*}
\sup_{\tau_0 < \tau \leq \infty} \int_\mathscr{T} |f|\, d\rho_{\Lambda,I,\tau} < \infty,
\end{equation*}
then
\begin{equation*}
\lim_{\tau \to \infty} \int_\mathscr{T} f\,d\rho_{\Lambda,I,\tau} = \int_\mathscr{T} f\,dm.
\end{equation*}
\end{lemma}

\begin{proof}
Proof of \cite[lemma 5.2]{Athreya2013-ql}, verbatim, with the measures $\rho_{\Lambda, I, \tau}$ replacing $\rho_{N}$.
\end{proof}

It should be noted that at this point we have the tools to generalize many Farey fraction statistical results of dynamical nature to primitive points of arbitrary lattices in $\RR^2$. As an example, we extend a proof of a theorem of \cite{Hall1984-bu} in \cite{Athreya2013-ql} from the denominators of Farey fractions to $x$-components of primitive lattice points.

\begin{corollary}
\label{corollary: Hall-Tanenbaum}
Let $\Lambda \subset \RR^2$ be a unimodular lattice, $I \subset \RR$ a finite interval, and $\mathcal{F}_I(\Lambda, \tau) = \left\{\begin{pmatrix}q_i \\ a_i\end{pmatrix}\right\}_{i=0}^{N_{\Lambda,I}(\tau) - 1}$, with the elements in increasing slope order of slope. Then for any $s,t \in \mathbb{C}$ with $\Re(s), \Re(t) \geq -1$
\begin{equation*}
\lim_{\tau \to \infty} \frac{1}{N_{\Lambda, I}(\tau) \tau^{s+t}} \sum_{i=0}^{N_{\Lambda,I}(\tau)-1} q_i^s q_{i+1}^t = \int_\mathscr{T} x^s y^t\, dm = 2 \int_\mathscr{T} x^s y^t\, dxdy.
\end{equation*}
\end{corollary}

\begin{proof}
Write $f_{s,t}(x,y) = x^sy^t$. For all $(x,y) \in \mathscr{T}$,
\begin{equation*}
|f_{s,t}(x,y)| \leq |f_{-1,-1}(x,y)|.
\end{equation*}
From \cref{lemma: sum of product of reciprocals}, we get
\begin{equation*}
\rho_{\Lambda, I, \tau}(f_{-1,-1}) = \frac{1}{N_{\Lambda, I}(\tau)} \sum_{i=0}^{N_{\Lambda,I}(\tau) - 1} \frac{\tau^2}{q_i q_{i+1}},
\end{equation*}
which is uniformly bounded by \cref{proposition: asymptotic growth of N_I}, and \cref{lemma: sum of product of reciprocals}. A direct calculation shows that $m(f_{-1,-1}) = \frac{\pi^2}{3}$, and so we are done by \cref{lemma: family of convergent integrals}.
\end{proof}

\subsection{Zippered Rectangle Decompositions of Tori, and Proof of Theorem \ref{theorem: continuous distribution arising from the three gap theorem}}
\label{subsection: proof of continuous distribution theorem}

We now have everything we need to prove \cref{theorem: continuous distribution arising from the three gap theorem}. We do this on two steps: First, in \cref{corollary: Zippered Rectangle Decompositions of Tori}, we describe the zippered rectangle decomposition of unimodular tori over unit length horizontals using primitive lattice points. Second, in \cref{proposition: change of variables}, we show how \cref{theorem: equidistribution of x-component measures} can be used to derive the continuous distribution in \cref{theorem: continuous distribution arising from the three gap theorem}. Beginning to end, this is a stand-alone proof of \cref{theorem: continuous distribution arising from the three gap theorem} using geometry and dynamics.

In \cref{proposition: change of variables}, the function $f$ is intended to be the aggregate function $f_z$ from \cref{subsection: gap distribution for circle rotations}. Integrating $f=f_z$ and the corresponding $F$ is direct, but lengthy, and can be extracted from the proof in \cite{Polanco_undated-ts} where \cref{theorem: continuous distribution arising from the three gap theorem} first showed up. For this particular $F$, the uniform boundedness of the integrals $\rho_{\Lambda,I,\tau}(F)$ follows from \cref{lemma: sum of product of reciprocals}, and the first two statements in \cref{lemma: Farey Neighbors}.

\begin{corollary}
\label{corollary: Zippered Rectangle Decompositions of Tori}
Let $\Lambda$ be a unimodular lattice, $\tau > 0$ with $\mathcal{F}(\Lambda, \tau) \neq \emptyset$, and $\begin{pmatrix}q_1 \\ a_1\end{pmatrix}, \begin{pmatrix}q_2 \\ a_2\end{pmatrix} \in \mathcal{F}(\Lambda, \tau)$ be two primitive vectors with consecutive slopes. Then for any $\alpha \in \RR$, if $\alpha \in \left(\frac{a_1}{q_1}, \frac{a_2}{q_2}\right)$, the zippered rectangle decomposition $\mathfrak{Z} = (\lambda_1, \lambda_2, \lambda_3, h_1, h_2, h_3)$ of the torus $\RR^2 / s_\frac{1}{\tau} h_\alpha \cdot \Lambda$ over a marked unit length horizontal is given by
\begin{eqnarray*}
\lambda_1 &=& 1 - \frac{q_1}{\tau}, \\
\lambda_2 &=& 1 - \lambda_2 - \lambda_3 = \frac{q_1 + q_2}{\tau} - 1, \\
\lambda_3 &=& 1 - \frac{q_2}{\tau}, \\
h_1 &=& \tau (\alpha q_1 - a_1), \\
h_2 &=& h_1 + h_3 = \tau ((a_2 - \alpha q_2) - (a_1 - \alpha q_1)), \text{ and}\\
h_3 &=& \tau (a_2 - \alpha q_2).
\end{eqnarray*}
\end{corollary}

\begin{proof}
\begin{figure}
\centering
\begin{tikzpicture}
\node at (3,1.5) {$\scriptstyle(q_2, a_2 - \alpha q_2)^T$};
\node at (3,1) {$\bullet$};
\node at (7,1) {$\times$};
\node at (0,0.4) {$\scriptstyle(0,0)^T$};
\node at (3.25,0.5) {$h_3$};
\node at (0,0) {$\bullet$};
\node at (4,0.2) {$\lambda_3$};
\node at (5,0.4) {$\scriptstyle(\tau, 0)^T$};
\node at (5,0) {$\times$};
\node at (5.25,-1) {$h_1$};
\node at (4,-2) {$\bullet$};
\node at (4,-2.6) {$\scriptstyle(q_1, a_1 - \alpha q_1)^T$};
\node at (4.5,-1.8) {$\lambda_1$};
\node at (8,-2) {$\times$};
\draw [->] (0,0) -- (6, 2);
\draw [->] (0,0) -- (6, 0);
\draw [->] (0,0) -- (6,-3);
\draw [line width=1] (3,1) -- (7,1);
\draw [line width=1] (0,0) -- (5,0);
\draw [line width=1] (4,-2) -- (8,-2);
\draw [->] (3,1) -- (3,0) ;
\draw [->] (5,0) -- (5,-2);
\end{tikzpicture}
\caption{A lift of the torus $\RR^2/h_\alpha\Lambda$ to the plane $\RR^2$ with a marked horizontal of length $\tau$. The left endpoints of the lifts of the said horizontal are marked as $\bullet$, and are situated at the lattice points $h_\alpha\Lambda$, and the right endpoints are marked as $\times$.}
\label{fig: zippered rectangles}
\end{figure}
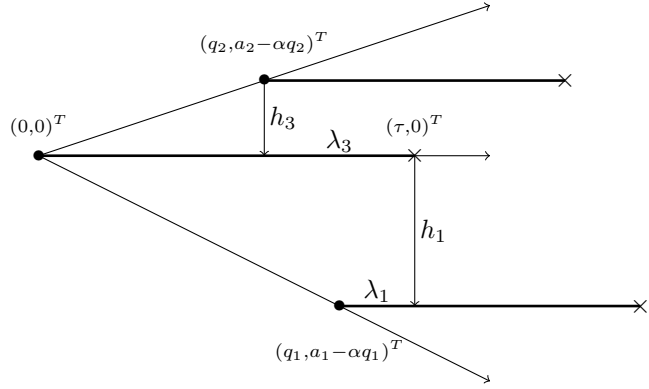

Lift the torus $\RR^2/h_\alpha \cdot \Lambda$ to the plane $\RR^2$ as in figure \ref{fig: zippered rectangles}. Note that the interior of the triangle bound by the lines $y = \left(\frac{a_1}{q_1} - \alpha\right)x$, $y = \left(\frac{a_2}{q_2} - \alpha\right)x$, and $x = \tau$ contains no lattice points from $h_\alpha \cdot \Lambda$. In particular, the interior of the quadrilateral $W_1$ with vertices at $\begin{pmatrix}0 \\ 0\end{pmatrix}$, $\begin{pmatrix}\tau \\ 0\end{pmatrix}$, $\begin{pmatrix}\tau \\ a_2 - \alpha q_2\end{pmatrix}$, and $\begin{pmatrix}q_2 \\ a_2 - \alpha q_2\end{pmatrix}$ contains no lattice points from $h_\alpha \cdot \Lambda$. Now, the interior of the quadrilateral $W_2$ with vertices at $\begin{pmatrix}0 \\ 0\end{pmatrix}$, $\begin{pmatrix}q_2 - \tau \\ 0\end{pmatrix}$, $\begin{pmatrix}q_2 - \tau \\ a_2 - \alpha q_2\end{pmatrix}$, and $\begin{pmatrix}q_2 \\ a_2 - \alpha q_2\end{pmatrix}$ contains no lattices points from $h_\alpha \Lambda$ as well. If there existed a $\vec{v}_0 \in \operatorname{int}(W_2) \cap \Lambda$, then the lattice point $\begin{pmatrix}q_2 \\ a_2 - \alpha q_2\end{pmatrix} - \vec{v}_0$ would be in the interior of $W_1$, which is a contradiction. This implies that no lift of the marked horizontal will extend between $\begin{pmatrix}q_2 \\ a_2 - \alpha q_2\end{pmatrix}$ and its projection of the lift of the marked horizontal situated at the origin. This gives $\lambda_3 = \tau - q_2$, and $h_3 = a_2 - \alpha q_2$. The same argument gives $\lambda_1 = \tau - q_1$, and $h_1 = \alpha q_1 - a_1$. The remaining parameters are given by $\lambda_2 = 1 - (\lambda_1 + \lambda_3)$, and $h_2 = h_1 + h_3$. This gives the zippered rectangle decomposition of $\RR^2/h_\alpha \cdot \Lambda$.

Scaling the zippered rectangle decomposition of $\RR^2/h_\alpha \Lambda$ by $s_\frac{1}{\tau}$ gives the canonical zippered rectangle decomposition of $\RR^2/s_\frac{1}{\tau}h_\alpha \cdot \Lambda$.
\end{proof}

\begin{proposition}
\label{proposition: change of variables}
Let $\Lambda$ be a unimodular lattice, and $I = [a,b] \subset \RR$ be a finite interval. Given a function $f : \RR^6 \to \RR$, if the function
\begin{equation*}
F(x,y) = \frac{1}{xy} \int_0^1 f(1 - x, x + y - 1, 1 - y, \frac{t}{y}, \left(\frac{1}{y} - \frac{1}{x}\right)t + \frac{1}{x}, \frac{1-t}{x})\, dt.
\end{equation*}
is defined on $\mathscr{T}$, is measureable, and its integral with respect to $\rho_{\Lambda, I, \tau}$ converges as $\tau \to \infty$ (i.e. $\rho_{\Lambda, I, \tau}(F) \to m(F)$), then
\begin{equation*}
\lim_{\tau \to \infty} \frac{1}{b-a}\int_a^b f(\mathfrak{Z}(s_\frac{1}{\tau}h_\alpha\cdot\Lambda)\,d\alpha = \frac{3}{\pi^2} \int_\mathscr{T} F\,dm = \frac{6}{\pi^2} \int_\mathscr{T} F(x,y)\,dxdy.
\end{equation*}
\end{proposition}

\begin{proof}
For any Farey arc $\left(\frac{a_i}{q_i}, \frac{a_{i+1}}{q_{i+1}}\right)$, the change of variables $\alpha = \frac{a_i}{q_i} + t\left(\frac{a_{i+1}}{q_{i+1}} - \frac{a_i}{q_i} \right) = \frac{a_i}{q_i} + \frac{t}{q_iq_{i+1}}$, with $t \in (0,1)$ gives
\begin{equation*}
\int_{a_i/q_i}^{a_{i+1}/q_{i+1}} f(\mathfrak{Z}(s_\frac{1}{\tau}h_\alpha \cdot \Lambda)) d\alpha = \frac{1}{\tau^2}F\left(\frac{q_1}{\tau}, \frac{q_2}{\tau}\right).
\end{equation*}
We thus get
\begin{eqnarray*}
\frac{1}{b-a}\int_a^b f(\mathfrak{Z}(s_\frac{1}{\tau}h_\alpha \cdot \Lambda)) d\alpha &=& \frac{1}{b-a} \sum_{i=0}^{N_{\Lambda,I}(\tau) - 1} \int_\frac{a_i}{q_i}^\frac{a_{i+1}}{q_{i+1}} f(\mathfrak{Z}(s_\frac{1}{\tau}h_\alpha \cdot \Lambda)) + o(1) \\
&=& \frac{1}{b-a} \sum_{i=0}^{N_{\Lambda,I}(\tau) - 1} \frac{1}{\tau^2} F\left(\frac{q_1}{\tau}, \frac{q_2}{\tau}\right) + o(1) \\
&=& \frac{1}{b-a} \frac{N_{\Lambda,I}(\tau)}{\tau^2} \frac{1}{N_{\Lambda,I}(\tau)} \sum_{i=0}^{N_{\Lambda,I}(\tau) - 1} F\left(\frac{q_1}{\tau}, \frac{q_2}{\tau}\right) + o(1) \\
&=& \frac{1}{b-a} \left(\frac{3(b-a)}{\pi^2} + o(1)\right) \rho_{\Lambda, I, \tau}(F) + o(1)
\end{eqnarray*}
From this follows that if $\rho_{\Lambda, I, \tau}(F) \to m(F)$, then
\begin{equation*}
\lim_{\tau \to \infty} \frac{1}{b-a}\int_a^b f(\mathfrak{Z}(s_\frac{1}{\tau}h_\alpha \cdot \Lambda)) d\alpha = \frac{3}{\pi^2} \int_\mathscr{T} F\,dm = \frac{6}{\pi^2} \int_\mathscr{T} F(x,y)\,dxdy.
\end{equation*}
\end{proof}

\section{Gap Theorem for $d$-IETS via Zippered Rectangles}
\label{section: gap theorem for d-iets via zippered rectangles}

In this section, we demonstrate that the technique used in \cref{section: the three gap theorem and zippered rectangles} to prove the Three Gap Theorem can be used to prove a generalization of the theorem for general intervals exchange transformations.

\subsection{The Surface $S_T$ Defined by IET $T$}
\label{subsection: the surface S_T defined by IET T}

Given an IET $T$, construct the surface $S_T$ (motivated by \cite{Athreya2012-yw}) as follows: Take the square $[0, 1] \times [0, 1]$, identify the horizontal sides by translation, and the vertical sides by $T$. Concretely, for each $i$, mark the points $(0, \alpha_i)$ on the left vertical, the points $(1, \beta_i)$ on the right vertical, and glue the vertical segments $[(0, \alpha_{i-1}), (0, \alpha_i)]$ and $[(1, \beta_{\pi(i)-1}), (1, \beta_{\pi(i)})]$ together.

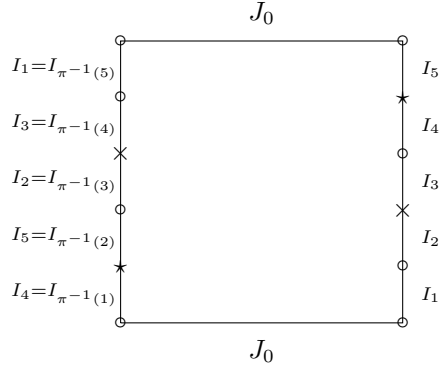
\begin{figure}
\centering
\begin{tikzpicture}[scale = .75]
\draw (0, 0) -- (5, 0) -- (5, 5) -- (0, 5) -- (0, 0);
\draw
node at (0, 0) {$\circ$}
node at (0, 1) {$\star$}
node at (0, 2) {$\circ$}
node at (0, 3) {$\times$}
node at (0, 4) {$\circ$}
node at (0, 5) {$\circ$}
node at (5, 0) {$\circ$}
node at (5, 1) {$\circ$}
node at (5, 2) {$\times$}
node at (5, 3) {$\circ$}
node at (5, 4) {$\star$}
node at (5, 5) {$\circ$};
\draw
node at (-1, 0.5) {$\scriptstyle I_4 = I_{\pi^{-1}(1)} $}
node at (-1, 1.5) {$\scriptstyle I_5 = I_{\pi^{-1}(2)}$}
node at (-1, 2.5) {$\scriptstyle I_2 = I_{\pi^{-1}(3)}$}
node at (-1, 3.5) {$\scriptstyle I_3 = I_{\pi^{-1}(4)}$}
node at (-1, 4.5) {$\scriptstyle I_1 = I_{\pi^{-1}(5)}$}
node at (5.5, 0.5) {$\scriptstyle I_1$}
node at (5.5, 1.5) {$\scriptstyle I_2$}
node at (5.5, 2.5) {$\scriptstyle I_3$}
node at (5.5, 3.5) {$\scriptstyle I_4$}
node at (5.5, 4.5) {$\scriptstyle I_5$};
\draw node at (2.5, -0.5) {$J_0$};
\draw node at (2.5, 5.5) {$J_0$};
\end{tikzpicture}
\caption{The surface $S_T$ for a $5$-IET that has combinatorial data $\pi = (1\ 5\ 2\ 3\ 4)$, with labels for gluing the line segments. The surface has three vertices $\circ, \star, \times$ with respective cone angles $6\pi, 2\pi, 2\pi$. It has genus $g = 2$, and lives in the stratum $\mathcal{H}_2(2)$.}
\end{figure}

After gluing, we mark the point on $S_T$ corresponding to $(0, 0)$ and call it a \textbf{marked origin}. As such, $S_T$ is a Riemann surface with a marked point. We intend to follow the same strategy used to prove the Three Gap Theorem: We start at the vertex $(0, 0)$, and flow horizontally for time $t = N$. A small subtlety arises here. For a general IET $T$, the horizontal trajectory in question can hit singularities in its path, and there will be a finite choice to make when it comes to flowing \emph{out} of each singularity. We choose the trajectory that gives the ``correct'' returns to the vertical through $(0, 0)$ (i.e. that intersects the vertical at the points corresponding to the orbit $\{T^n0\}_{n=0}^\infty$). We call this the \textbf{canonical horizontal trajectory of length $N$ coming out of the origin}.

The following lemma characterizes the property of marked origins being singularities.

\begin{lemma}
\label{lemma: characterization of a marked origin being a singularity}
Let $T$ and $S_T$ be as above. The marked origin is a singularity if and only if $\pi^{-1}(1) - \pi^{-1}(d) \neq 1$.
\end{lemma}

\begin{proof}
Since the permutation $\pi$ is irreducible, it is easy to see that $(0,0)$ is not a singularity if and only if the first and last segment $I_{\pi^{-1}(1)}$ and $I_{\pi^{-1}(d)}$ on the left side correspond to segments on the right side that are \emph{back to back}, that is, $\pi^{-1}(1) = \pi^{-1}(d) + 1$.
\end{proof}

We call IETs with $\pi^{-1}(1) - \pi^{-1}(d) = 1$ \textbf{arc exchange maps}.

\begin{lemma}
Let $T$ and $S_T$ be as above.
\begin{enumerate}
\item The first return map of the horizontal flow to any closed vertical segment is given by the map $T$.
\item The IET $T$ fails to satisfy the i.d.o.c condition if and only if there is a horizontal saddle connection that is not the horizontal unit length closed line segment $[(0, \alpha_0), (1, \beta_0)] \sim [(0, \alpha_d), (1, \beta_d)]$.
\end{enumerate}
\end{lemma}

\begin{lemma}
Let $T$ and $S_T$ be as above. For any $\alpha \in \RR$:
\begin{enumerate}
\item The construction $S_\cdot$ is compatible with shearing by $h_\alpha$. That is, $h_\alpha \cdot S_T = S_{T \circ R_\alpha}$.
\item Shearing by $h_\alpha$ takes any line segment $X \subset S_T$ of slope $\alpha$ to a horizontal line segment $h_\alpha \cdot X \subset S_{T \circ R_\alpha}$. Moreover, the return time of any $x \in X$ to $X$ by the vertical flow on $S_T$ is the same as that of $h_\alpha \cdot x \in h_\alpha \cdot X$ to $h_\alpha \cdot X$ by the vertical flow on $S_{T \circ R_\alpha}$.
\end{enumerate}
\end{lemma}

This proves the following property of the compositions of IETs and circle rotations.

\begin{corollary}
For any IET $T$, there are at most countably many $\alpha \in \RR$ for which $T \circ R_\alpha$ does not satisfy the i.d.o.c condition.
\end{corollary}

\begin{proof}
The failure of $T \circ R_\alpha$ to satisfy the i.d.o.c is equivalent to the existence of non-horizontal saddle connections in $S_{T \circ R_\alpha}$ that are not $[(0, \alpha_0), (1, \beta_0)] \sim [(0, \alpha_d), (1, \beta_d)]$. Any horizontal saddle connection in $S_{T \circ R_\alpha} = h_\alpha \cdot S_T$ corresponds to a saddle connection in $S_T$ with slope $\alpha$. The result follows from the fact that there are countably many saddle connections on any translation surface.
\end{proof}

\subsection{Gap Theorem for IETs}
\label{subsection: gap distribution for IETs}

The argument used in \cref{section: the three gap theorem and zippered rectangles} can be adapted to $S_T$ to prove the following.

\begin{theorem}
For a fixed IET $T$ satisfying the i.d.o.c, and integer $N \geq 1$, the pairs \[(\text{gap length}, \text{number of gaps of that length})\] describing the multiset $\widetilde{\operatorname{Gaps}}_{T, N}$ correspond to the height-width parameters $(h, \lambda)$ of the zippered rectangle decomposition of the surface $S_T$ over the canonical horizontal trajectory of length $N$ coming out of the marked origin.
\end{theorem}

The above, along with \cref{proposition: number of discontinuities of first return map}, and the characterization of a marked origin being a singularity from \cref{lemma: characterization of a marked origin being a singularity} proves \cref{theorem: d + 2 gap theorem}.

\subsection{Gap Distribution for Compositions of IETs and Circle Rotations}

For a general IET $T$, we denote the number of \textbf{normalized gap lengths} greater than or equal to $z$ by
\begin{eqnarray*}
\mathcal{G}_{T, N}(z) &:=& \# \{\ell \in \widetilde{\operatorname{Gaps}}_{T, N} \mid \ell \geq \frac{z}{N}\} \\
&=& \#\{L \in N \cdot \widetilde{\operatorname{Gaps}}_{T, N} \mid L \geq z\}
\end{eqnarray*}

For compositions of $T$ and circle rotations, we can consider the \textbf{average gap distribution}
\begin{equation*}
    g_{T \circ R_\cdot}^{[a, b]}(z; N) := \frac{1}{b-a}\int_a^b \frac{\mathcal{G}_{T \circ R_\cdot, N}(z)}{N}\,d\alpha.
\end{equation*}

Copying \cref{subsection: notation}, we define an operator $\mathfrak{Z}_\text{canon}$ as follows: for any translation surface $S$ with a marked point, and a canonical choice of horizontal unit length trajectory coming out of the marked point, $\mathfrak{Z}_\text{canon}(S)$ is the vector of length-height data of the zippered rectangle decomposition of $S$ over the aforementioned horizontal.

An aggregate function $f_z$ can appropriately be defined similar to \cref{subsection: gap distribution for circle rotations}.

The argument used to prove \cref{theorem: distribution for TGT} can adapted to relate the average gap distribution forcompositions of a fixed IET $T$ and circle rotations to the distribution of zippered rectangle decomposition
\begin{equation*}
    g_{T \circ R_\cdot}^{[a, b]}(z; N) = \frac{1}{b-a} \int_a^b f_z(\mathfrak{Z}_\text{canon}(g_{\log N}h_\alpha \cdot S_T))\,d\alpha.
\end{equation*}
This proves the \cref{theorem: axiomatic theorem}.

\section{Gap Theorem for $d$-iets via Graph Theory}
\label{section: gap theorem for d-iets via graph theory}
In this section, we present a self-contained proof for \cref{theorem: d + 2 gap theorem} motivated by the Rauzy graph approach presented below. We use the same notation for IETs, discontinuities, gap sets, and gap multisets from before.

\subsection{The Rauzy Graph Approach}
\label{subsection: the Rauzy graph approach}

\begin{theorem}[$3(d-1)$ Gap Theorem, \cite{Boshernitzan1985-fb}]
\label{theorem: 3(d-1) gap theorem)}
Let $T$ be a minimal $d$-IET. Then
\begin{equation*}
    \# \operatorname{Gaps}_{T, N} \leq 3(d-1).
\end{equation*}
\end{theorem}

When $d = 2$, the bounds in \cref{theorem: d + 2 gap theorem}, and \cref{theorem: 3(d-1) gap theorem)} agree with that of the Three Gap Theorem (\cref{theorem: Three Gap Theorem}).

Following \cite{Boshernitzan1985-fb}, define a directed graph $\operatorname{GGaps}_{T, N} = (V_{T, N}, E_{T,N})$ as follows: The vertices agree with the gaps $V_{T,N} = \widetilde{\operatorname{Gaps}}_{T, N}$, and there exists an edge $e_{i,j} : \operatorname{gap}_{T, N}(i) \to \operatorname{gap}_{T, N}(j)$ if and only if $(T^{-1}\operatorname{gap}_{T,N}(i)) \cap \operatorname{gap}_{T,N}(j) \neq \varnothing$. It should be noted that $\operatorname{GGaps}_{T,N}$ is the \textbf{Rauzy graph} describing the coding generated by $T^{-1}$ and the partioning of $[0,1)$ defined by $\{T^n0\}_{n=0}^{N-1}$.

For a general directed graph $G = (V, E)$, a function $w : V \sqcup E \to \RR^+$ is said to be a \textbf{weight function} if $\sum_{v \in V} w(v) = 1$, and for any $v \in V$, \[w(v) = \sum_{e \in \operatorname{in}(v)}w(e) = \sum_{e \in \operatorname{out}(v)}w(e),\] where $\operatorname{in}(v)$ and $\operatorname{out}(v)$ are the edges of $G$ going in and coming out of $v$. The \textbf{support} $\operatorname{supp}(w)$ of $w$is the subgraph $G_w = (V_w, E_w) \subseteq G$ defined by $V_w = \{v \in V \mid w(v) > 0\}$ and $E_w = \{e \in E \mid w(e) > 0\}$.

A cycle $v_1, v_2, \cdots, v_k$ is said to be a \textbf{distinct cycle} if $\operatorname{indeg}(v_i) = \operatorname{outdeg}(v_i)$ for every $i = 1, 2, \cdots, k$. The following is then true.

\begin{proposition}[\cite{Boshernitzan1985-fb}]
Let $G$ be a directed graph without distinct cycles, and $w$ a weight function of $G$ supported on the whole graph. The cardinality of the set of possible vertex weights has the upper bound \[\#\{w(v) \mid v \in V\} \leq 3 (\#E - \#V).\]
\end{proposition}

The Lebesgue measure induces a weight function $w$ on the graph $\operatorname{GGaps}_{T, N}$: for every $v_i \in V_{T, N}$, take $w(v_i) = \operatorname{Leb}(\operatorname{gap}_{T,N}(i))$, and for every $e_{i, j} \in E_{T,N}$, take $w(e_{i, j}) = \operatorname{Leb}((T^{-1}\operatorname{gap}_{T,N}(i)) \cap \operatorname{gap}_{T,N}(j))$.

The outdegree of a vertex $v_i \in \operatorname{GGaps}_{T, N}$ can be expressed as
\[\operatorname{outdeg}(v_i) = 1 + 1_{\operatorname{gap}_{T,N}(i)}(T^N0) + \sum_{i=1}^{d-1}1_{\operatorname{gap}_{T,N}(i)}(\alpha_i).\] From this it follows that
\begin{eqnarray*}
\#E_{T,N} - \#V_{T,N} &=& \sum_{v \in V} (\operatorname{outdeg}(v) - 1) \\
&=& \sum_{i=0}^{d-1} \left(1_{\operatorname{gap}_{T,N}(i)}(T^N0) + \sum_{k=1}^{d-1}1_{\operatorname{gap}_{T,N}(i)}(\alpha_k)\right) \\
&=& d - 1.
\end{eqnarray*}
The reason this is $d - 1$ and not $d$ is that $T^10 = \alpha_{i_0}$, with $i_0 = \pi(1) - 1$. That is, the discontinuity $\alpha_{i_0}$ is an endpoint of one of the gaps, and hence necessarily does not contribute to the sum.

A minimal IET $T$ will not have distinct cycles in $\operatorname{GGaps}_{T, N}$. This proves the $3(d-1)$ Gap Theorem (\cref{theorem: 3(d-1) gap theorem)}).

\subsection{Modifying the Rauzy Graphs $\operatorname{GGaps}_{T, N}$}

In this section, we modify the Rauzy graphs $\operatorname{GGaps}_{T, N}$ that were used in \cref{subsubsection: gaps of iets} to prove \cref{theorem: 3(d-1) gap theorem)}.

We construct graphs $\operatorname{FGaps}_{T, N}$ to replace the Rauzy graphs $\operatorname{GGaps}_{T, N}$ as follows: For $i=0, 1, \cdots, d-1$, we denote the closest orbit point in the orbit segment $\{T^n0\}_{n=0}^{N-1}$ to a discontinuity $\beta_i$ of $T$ from the right by
\[r(i) = r_{T, N}(i) := \min \{T^n0 \mid 1 \leq n < N, T^n0 > \beta_i\},\]
and for $i = 1, 2, \cdots, d$, we denote the closest orbit point to a discontinuity $\beta_i$ from the left by
\[l(i) = l_{T, N}(i) := \max \{T^n0 \mid 1 \leq n < N, T^n0 < \beta_i\}.\] We also denote the collection of ``slots'' on the right and left of the discontinuities of $T$ by
\begin{equation*}
    \mathfrak{R} = \mathfrak{R}_{T,N} := \{(\beta_i, r(i)) \mid i = 0, 1, \cdots, d-1\},
\end{equation*}
and
\begin{equation*}
    \mathfrak{L} = \mathfrak{L}_{T,N} := \{(l(i), \beta_i) \mid i = 1, 2, \cdots, d\}.
\end{equation*}
Finally, we write
\begin{equation*}
    R_i = R_{T, N, i} := r(i) - \beta_i,\ i = 0, 1, \cdots, d-1
\end{equation*}
and
\begin{equation*}
    L_i = L_{T, N, i} := \beta_i - l(i),\ i = 1, 2, \cdots, d.
\end{equation*}
for the respective lengths of the right and left slots. As was done in \cref{subsubsection: gaps of iets}, we \emph{notationally} identify $R_i$ with $(\beta_i, r(i))$, and $L_i$ with $(l(i), \beta_i)$.

We define the graph $\operatorname{FGaps}_{T, N}$ as follows:
\begin{itemize}
\item Vertices: $FV_{T, N} = \widetilde{Gaps}_{T, N} \cup \mathfrak{R}_{T,N} \cup \mathfrak{L}_{T,N}$.
\item Edges: When $T^{-1}$ acts on a gap in $\widetilde{\operatorname{Gaps}}_{T, N}$, it maps it to a disjoint union---modulo the endpoints---of elements of $\widetilde{Gaps}_{T, N} \cup \mathfrak{R}_{T,N} \cup \mathfrak{L}_{T,N}$. A directed edge goes out of each gap and into the elements of $\widetilde{Gaps}_{T, N} \cup \mathfrak{R}_{T,N} \cup \mathfrak{L}_{T,N}$ when acted on by $T^{-1}$. The collection of all edges is denoted by $FE_{T,N}$.
\end{itemize}

It can be easily seen that $\operatorname{FGaps}_{T, N}$ is a forest. It should also be noted that if we glue each pair $L_i$ and $R_i$, $i = 1, 2, \cdots, d-1$, together and identify them with the gap $L_i \cup \{\beta_i\} \cup R_i$ surrounding the discontinuity $\beta_i$, we get the Rauzy graph $\operatorname{GGaps}_{T, N}$.

We first present an example to make both the definition of the forest and the idea of the proof that will follow clearer.

\subsection{Example}

Consider the IET with length data $\mathbf{\lambda} = (1/\sqrt{3}, 1/\sqrt{2} - 1/\sqrt{3}, 1 - 1/\sqrt{2})$ (i.e. $T$ has discontinuities at $\beta_1 = 1/\sqrt{3}$ and $\beta_2 = 1/\sqrt{2}$), and combinatorial data $\pi = \begin{pmatrix} 1 & 2 & 3 \\ 3 & 2 & 1\end{pmatrix}$. The inverse map $T^{-1}$ is an IET, with length data $\mathbf{\lambda}^{-1} = (1 - 1/\sqrt{2}, 1/\sqrt{2} - 1/\sqrt{3}, 1/\sqrt{3})$ (i.e. discontinuities at $\alpha_1 = 1 - 1/\sqrt{2}$ and $\alpha_2 = 1 - 1/\sqrt{3}$), and combinatorial data $\pi^{-1} = \pi$.

Below we show two copies of the interval $[0,1)$, along with the orbit points $\{T^k 0\}_{k=1}^8$, the points $\alpha_i$ bounding the subintervals being permuted by $T^{-1}$, and the points $\beta_i$ bounding the subintervals being permuted by $T$.

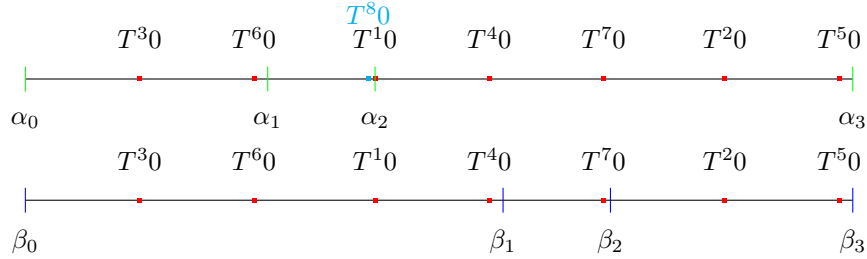
\begin{center}
\begin{tikzpicture}[scale = 11]
\draw (0,0) -- (1,0)

node[fill=red, scale = 0.25] at (0.42265,0) {}
node at (0.42265,0.05) {$T^{1}0$}

node[fill=red, scale = 0.25] at (0.845299,0) {}
node at (0.845299,0.05) {$T^{2}0$}

node[fill=red, scale = 0.25] at (0.138193,0) {}
node at (0.138193,0.05) {$T^{3}0$}

node[fill=red, scale = 0.25] at (0.560842,0) {}
node at (0.560842,0.05) {$T^{4}0$}

node[fill=red, scale = 0.25] at (0.983492,0) {}
node at (0.983492,0.05) {$T^{5}0$}

node[fill=red, scale = 0.25] at (0.276385,0) {}
node at (0.276385,0.05) {$T^{6}0$}

node[fill=red, scale = 0.25] at (0.699035,0) {}
node at (0.699035,0.05) {$T^{7}0$}

node[fill=cyan, scale = 0.25] at (0.414578,0) {}
node at (0.414578,0.08) {$\textcolor{cyan}{T^{8}0}$};

\draw [green]
(0,-0.015) -- (0,0.015)
(0.292893, -0.015) -- (0.292893, 0.015)
(0.42265, -0.015) -- (0.42265, 0.015)
(1,-0.015) -- (1, 0.015);
\draw
node at (0,-0.05) {$\alpha_0$}
node at (0.292893,-0.05) {$\alpha_1$}
node at (0.42265,-0.05) {$\alpha_2$}
node at (1,-0.05) {$\alpha_3$};
\end{tikzpicture}

\begin{tikzpicture}[scale = 11]
\draw (0,0) -- (1,0)

node[fill=red, scale = 0.25] at (0.42265,0) {}
node at (0.42265,0.05) {$T^{1}0$}

node[fill=red, scale = 0.25] at (0.845299,0) {}
node at (0.845299,0.05) {$T^{2}0$}

node[fill=red, scale = 0.25] at (0.138193,0) {}
node at (0.138193,0.05) {$T^{3}0$}

node[fill=red, scale = 0.25] at (0.560842,0) {}
node at (0.560842,0.05) {$T^{4}0$}

node[fill=red, scale = 0.25] at (0.983492,0) {}
node at (0.983492,0.05) {$T^{5}0$}

node[fill=red, scale = 0.25] at (0.276385,0) {}
node at (0.276385,0.05) {$T^{6}0$}

node[fill=red, scale = 0.25] at (0.699035,0) {}
node at (0.699035,0.05) {$T^{7}0$};

\draw [blue]
(0,-0.015) -- (0,0.015)
(0.57735, -0.015) -- (0.57735, 0.015)
(0.707107, -0.015) -- (0.707107, 0.015)
(1,-0.015) -- (1, 0.015);
\draw
node at (0,-0.05) {$\beta_0$}
node at (0.57735,-0.05) {$\beta_1$}
node at (0.707107,-0.05) {$\beta_2$}
node at (1,-0.05) {$\beta_3$};
\end{tikzpicture}
\end{center}

Now, considering the action of $T^{-1}$ on $\widetilde{\operatorname{Gaps}}_{T, 8}$, we get the following diagram for the forest $\operatorname{FGaps}_{T, 8}$.
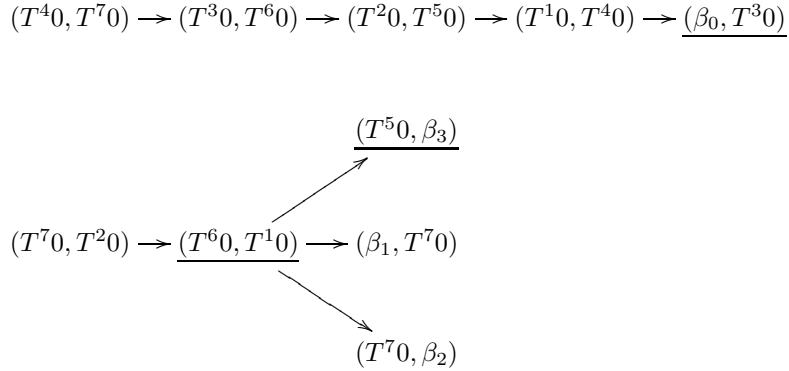
\begin{displaymath}
\xymatrix@C-1pc{
(T^4 0, T^7 0) \ar[r] & (T^3 0, T^6 0) \ar[r] & (T^2 0, T^5 0) \ar[r] & (T^1 0, T^4 0) \ar[r] & \underline{(\beta_0, T^3 0)} \\
 & & \underline{(T^5 0, \beta_3)} & & \\
(T^7 0, T^2 0) \ar[r] & \underline{(T^6 0, T^1 0)} \ar[ur] \ar[r] \ar[dr] & (\beta_1, T^7 0) & & \\
 & & (T^7 0, \beta_2) & &
}
\end{displaymath}
When we get a gap pointing to another gap like
\begin{displaymath}
\xymatrix{
(T^3 0, T^6 0) \ar[r] & (T^2 0, T^5 0)
}
\end{displaymath}
the two gaps have the same length. If a gap splits (i.e. its preimage under $T^{-1}$ is a disjoint union of more than one interval), its length is the sum of the intervals it splits into. As such, the gaps in the above diagram that contribute to the set of gap lengths $\operatorname{Gaps}_{T, 8}$ ar the gaps that show on the extreme right, namely $(\beta_0, T^3 0)$ and $(T^5 0, \beta_3)$, and the gaps that split, namely $(T^6 0, T^1 0)$. Note that the intervals showing on the extreme right in the above diagram are all slots surrounding the discontinuitites $\beta_i$ of $T$.

We thus get that \[\operatorname{Gaps}_{T, 8} := \{R_0, L_3, L_3 + R_1 + L_2\}.\]

\subsection{Combinatorial Proof of Theorem \ref{theorem: d + 2 gap theorem}}

We consider any orbit $\{T^n0\}_{n=0}^{N-1}$, with $N$ an integer large enough so that the orbit points separate the discontinuities $\alpha_i$ of $T^{-1}$. For counting purposes, we have to consider:
\begin{enumerate}
\item in \cref{subsection: gaps that split}: the gaps that intersect with the discontinuities of $T^{-1}$, or the point $T^N 0$, and so their images under $T^{-1}$ split, and
\item in \cref{subsection: gaps that get pulled to other gaps}: the gaps whose images under $T^{-1}$ are other gaps, and
\item in \cref{subsection: the first and last gaps}: the first and last gaps.
\end{enumerate}

We call the point $T^N0$ a \textbf{ghost orbit point}, since it is not in the orbit segment $\{T^n0\}_{n=0}^{N-1}$, but introduces the orbit point $T^{N-1}0$ when it is acted on by $T^{-1}$.

We enumerate the different cases pictorially, with red representing orbit points, and cyan representing the point $T^N 0$. The index $i_0 = \pi(1) - 1$ specifies the discontinuity $\alpha_{i_0}$ where $T^1 0$ occurs.

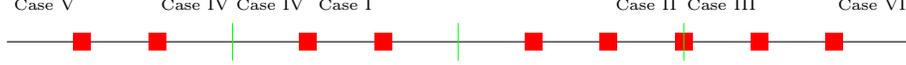
\begin{center}
\begin{tikzpicture}
\draw (0,0) -- (12,0)
node at (0.5, 0.5) {\tiny Case V}
node[fill=red] at (1, 0) {}
node[fill=red] at (2, 0) {}
node at (2.5, 0.5) {\tiny Case IV}
node at (3.5, 0.5) {\tiny Case IV}
node[fill=red] at (4, 0) {}
node at (4.5, 0.5) {\tiny Case I}
node[fill=red] at (5, 0) {}
node[fill=red] at (7, 0) {}
node[fill=red] at (8, 0) {}
node at (8.5, 0.5) {\tiny Case II}
node[fill=red] at (9,0) {}
node at (9.5, 0.5) {\tiny Case III}
node at (11.5, 0.5) {\tiny Case VI}
node[fill=red] at (10, 0) {}
node[fill=red] at (11, 0) {};
\draw[green] (3,-0.25) -- (3, 0.25);
\draw[green] (6,-0.25) -- (6, 0.25);
\draw[green] (9,-0.25) -- (9, 0.25);
\end{tikzpicture}
\end{center}

Based on the location where the ghost orbit point $T^N0$ situates itself, there are six distinct cases that should be accounted for.
\begin{itemize}
\item Case I: The point $T^N0$ belongs to a gap $(T^{n_1}0, T^{n_2}0)$ with $n_1, n_2 \neq 1$.
\item Case II: Same as case I with $n_2 = 1$.
\item Case III: Same as case I with $n_1 = 1$.
\item Case IV: The point $T^N0$ belongs to $(T^{n_1}0, T^{n_2}0)$ that includes a $\alpha_i$.
\item Case V: The point $T^n 0$ belongs to the first gap $(0, T^{\sigma(1)}0)$.
\item Case VI: The point $T^N0$ belongs to the last gap $(T^{\sigma(N-1)}0, 1)$.
\end{itemize}

We now proceed to count the possible number of gap lengths.

\subsection{Gaps that Split}
\label{subsection: gaps that split}

\subsubsection{Case I}

\begin{itemize}

\item One instance of

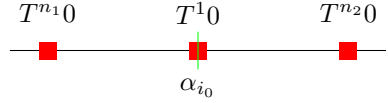
\begin{center}
\begin{tikzpicture}
\draw (-0.5,0) -- (4.5,0) node[fill=red] at (0,0) {} node at (0, 0.5) {$T^{n_1}0$} node[fill=red] at (2,0) {} node at (2,0.5) {$T^1 0$} node at (2,-0.5) {$\alpha_{i_0}$} node[fill=red] at (4,0) {} node at (4,0.5) {$T^{n_2}0$};
\draw[green] (2,-0.25) -- (2, 0.25);
\end{tikzpicture}
\end{center}

\begin{equation}
\label{eq:I1}
\tag{I.1}
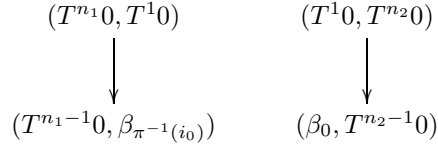
\xymatrix{
(T^{n_1}0, T^10) \ar[d]    & (T^1 0, T^{n_2}0) \ar[d] \\
(T^{n_1 - 1}0, \beta_{\pi^{-1}(i_0)})    & (\beta_0, T^{n_2 - 1} 0)
}
\end{equation}

This contributes the two lengths $L_{\pi^{-1}(i_0)}$ and $R_0$ to the set of gap lengths.

\item $d-2$ instances of

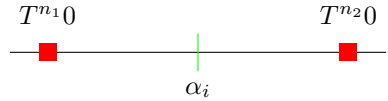
\begin{center}
\begin{tikzpicture}
\draw (-0.5,0) -- (4.5,0) node[fill=red] at (0,0) {}  node at (0, 0.5) {$T^{n_1}0$} node at (2,-0.5) {$\alpha_i$} node[fill=red] at (4,0) {} node at (4,0.5) {$T^{n_2}0$};
\draw[green] (2,-0.25) -- (2, 0.25);
\end{tikzpicture}
\end{center}

\begin{equation}
\label{eq:I2}
\tag{I.2}
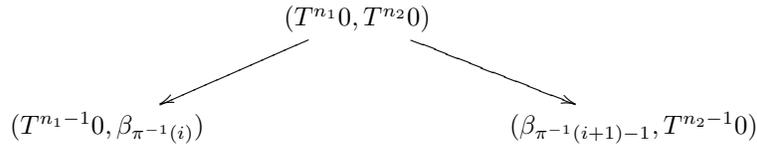
\xymatrix{
    &  (T^{n_1}0, T^{n_2}0) \ar[dl] \ar[dr]  &   \\
(T^{n_1 - 1}0, \beta_{\pi^{-1}(i)}) & & (\beta_{\pi^{-1}(i+1) - 1}, T^{n_2 - 1}0)
}
\end{equation}

Each of those instances contributes $L_{\pi^{-1}(i)} + R_{\pi^{-1}(i+1) - 1}$ to the set of gap lengths, for a total of $d-2$ elements.

\item One instance of

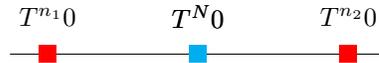
\begin{center}
\begin{tikzpicture}
\draw (-0.5,0) -- (4.5,0) node[fill=red] at (0,0) {} node at (0, 0.5) {$T^{n_1}0$} node[fill=cyan] at (2,0) {} node at (2, 0.5) {$T^N0$} node at (2,0.5) {$T^N 0$} node[fill=red] at (4,0) {} node at (4,0.5) {$T^{n_2}0$};
\end{tikzpicture}
\end{center}

\begin{equation}
\label{eq:I3}
\tag{I.3}
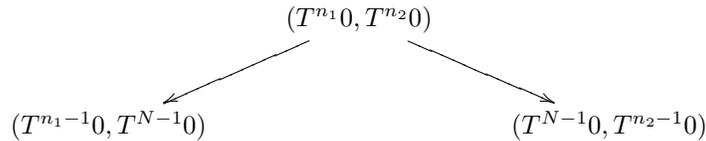
\xymatrix{
 & (T^{n_1}0, T^{n_2}0) \ar[dl] \ar[dr] & \\
(T^{n_1 - 1}0, T^{N-1}0) & & (T^{N-1}0, T^{n_2 - 1}0)
}
\end{equation}

This contributes the sum of the lengths of the two consecutive gaps $(T^{n_1 - 1}0, T^{N-1}0)$ and $(T^{N-1}0, T^{n_2 - 1}0)$ to the set of gap lengths.
\end{itemize}

\subsubsection{Case II}

\begin{itemize}

\item One instance of

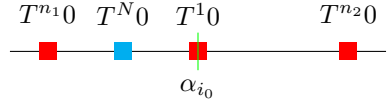
\begin{center}
\begin{tikzpicture}
\draw (-0.5,0) -- (4.5,0) node[fill=red] at (0,0) {} node at (0, 0.5) {$T^{n_1}0$}  node[fill=cyan] at (1,0) {} node at (1, 0.5) {$T^N0$} node[fill=red] at (2,0) {} node at (2,0.5) {$T^1 0$} node at (2,-0.5) {$\alpha_{i_0}$} node[fill=red] at (4,0) {} node at (4,0.5) {$T^{n_2}0$};
\draw[green] (2,-0.25) -- (2, 0.25);
\end{tikzpicture}
\end{center}

\begin{equation}
\label{eq:II1}
\tag{II.1}
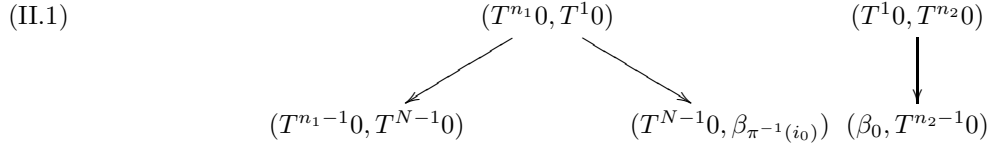
\xymatrix@C-2pc{
    & (T^{n_1}0, T^10) \ar[dl] \ar[dr] & & (T^1 0, T^{n_2}0) \ar[d] \\
(T^{n_1 - 1}0, T^{N-1}0) &    & (T^{N-1}0, \beta_{\pi^{-1}(i_0)}) & (\beta_0, T^{n_2 - 1} 0)
}
\end{equation}

This instance contributes the two lengths $|(T^{n_1 - 1}0, T^{N-1}0)| + L_{\pi^{-1}(i_0)}$ and $R_0$ to the set of gap lengths.

\item $d-2$ instances of

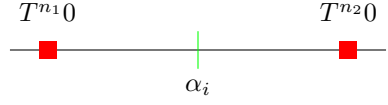
\begin{center}
\begin{tikzpicture}
\draw (-0.5,0) -- (4.5,0) node[fill=red] at (0,0) {} node at (0, 0.5) {$T^{n_1}0$} node at (2,-0.5) {$\alpha_i$} node[fill=red] at (4,0) {} node at (4,0.5) {$T^{n_2}0$};
\draw[green] (2,-0.25) -- (2, 0.25);
\end{tikzpicture}
\end{center}

\begin{equation}
\label{eq:II2}
\tag{II.2}
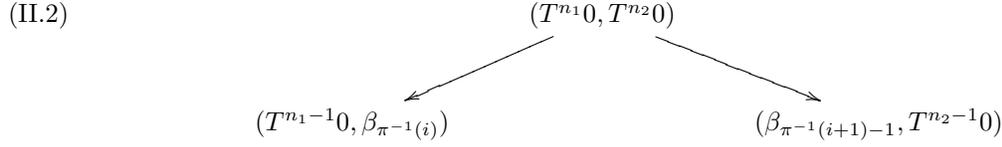
\xymatrix{
    & (T^{n_1}0, T^{n_2}0) \ar[dr] \ar[dl] & \\
(T^{n_1 - 1}0, \beta_{\pi^{-1}(i)}) &   & (\beta_{\pi^{-1}(i+1) - 1}, T^{n_2 - 1}0)
}
\end{equation}

Each of these instance contribute $L_{\pi^{-1}(i)} + R_{\pi^{-1}(i+1) - 1}$ to the set of gap lengths, for a total of $d-2$ lengths.

\end{itemize}

\subsubsection{Case III}

\begin{itemize}

\item One instance of

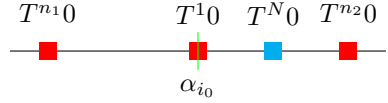
\begin{center}
\begin{tikzpicture}
\draw (-0.5,0) -- (4.5,0) node[fill=red] at (0,0) {} node at (0, 0.5) {$T^{n_1}0$} node[fill=red] at (2,0) {} node at (2,-0.5) {$\alpha_{i_0}$} node at (2,0.5) {$T^1 0$}   node[fill=cyan] at (3,0) {} node at (3, 0.5) {$T^N0$} node[fill=red] at (4,0) {} node at (4,0.5) {$T^{n_2}0$};
\draw[green] (2,-0.25) -- (2, 0.25);
\end{tikzpicture}
\end{center}

\begin{equation}
\label{eq:III1}
\tag{III.1}
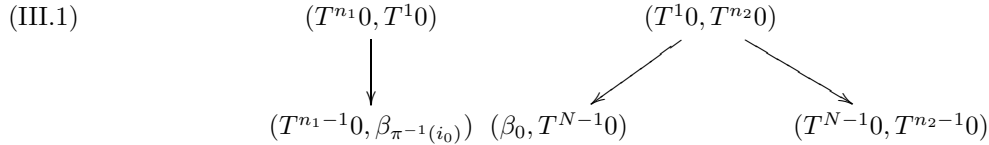
\xymatrix@C-2pc{
(T^{n_1}0, T^1 0) \ar[d] & & (T^1 0, T^{n_2}0) \ar[dl] \ar[dr] \\
(T^{n_1 - 1}0, \beta_{\pi^{-1}(i_0)}) & (\beta_0, T^{N-1} 0) & & (T^{N-1} 0, T^{n_2 - 1}0)
}
\end{equation}

\item $d-2$ instances of

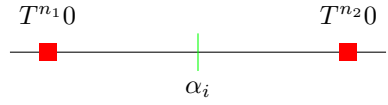
\begin{center}
\begin{tikzpicture}
\draw (-0.5,0) -- (4.5,0) node[fill=red] at (0,0) {} node at (0, 0.5) {$T^{n_1}0$} node at (2,-0.5) {$\alpha_i$} node[fill=red] at (4,0) {} node at (4,0.5) {$T^{n_2}0$};
\draw[green] (2,-0.25) -- (2, 0.25);
\end{tikzpicture}
\end{center}

\begin{equation}
\label{eq:III2}
\tag{III.2}
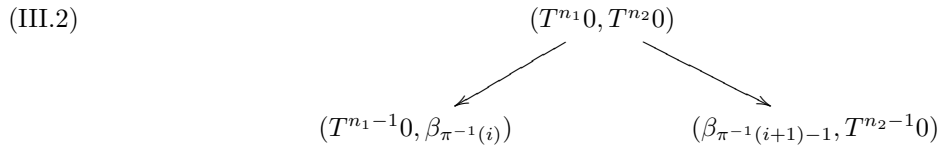
\xymatrix@C-2pc{
    & (T^{n_1}0, T^{n_2}0) \ar[dr] \ar[dl] & \\
(T^{n_1 - 1}0, \beta_{\pi^{-1}(i)}) &   & (\beta_{\pi^{-1}(i+1) - 1}, T^{n_2 - 1}0)
}
\end{equation}

Each of those instances contribute $L_{\pi^{-1}(i)} + R_{\pi^{-1}(i+1) - 1}$ to the set of gap lengths, for a total of $d-2$ lengths.

\end{itemize}

\subsubsection{Case IV}

\begin{itemize}

\item One instance of

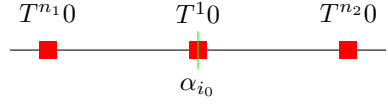
\begin{center}
\begin{tikzpicture}
\draw (-0.5,0) -- (4.5,0) node[fill=red] at (0,0) {} node at (0, 0.5) {$T^{n_1}0$} node[fill=red] at (2,0) {} node at (2,0.5) {$T^1 0$} node at (2,-0.5) {$\alpha_{i_0}$} node[fill=red] at (4,0) {} node at (4,0.5) {$T^{n_2}0$};
\draw[green] (2,-0.25) -- (2, 0.25);
\end{tikzpicture}
\end{center}

\begin{equation}
\label{eq:IV1}
\tag{IV.1}
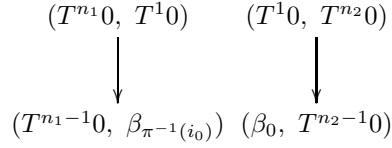
\xymatrix@C-2pc{
(T^{n_1}0,\ T^10) \ar[d] & (T^1 0,\ T^{n_2}0) \ar[d] \\
(T^{n_1 - 1}0,\ \beta_{\pi^{-1}(i_0)}) & (\beta_0,\ T^{n_2 - 1} 0)
}
\end{equation}

This contributes two lengths $L_{\pi^{-1}(i_0)}$ and $R_0$ to the set of gap lengths.

\item One instance of

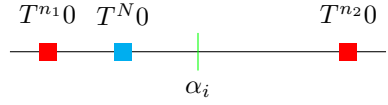
\begin{center}
\begin{tikzpicture}
\draw (-0.5,0) -- (4.5,0) node[fill=red] at (0,0) {} node at (0, 0.5) {$T^{n_1}0$} node at (2,-0.5) {$\alpha_i$} node[fill=cyan] at (1,0) {} node at (1, 0.5) {$T^N0$} node[fill=red] at (4,0) {} node at (4,0.5) {$T^{n_2}0$};
\draw[green] (2,-0.25) -- (2, 0.25);
\end{tikzpicture}
\end{center}

\begin{equation}
\label{eq:IV2}
\tag{IV.2}
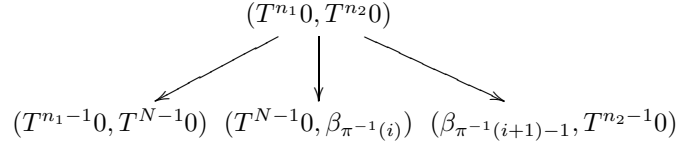
\xymatrix@C-2pc{
    & (T^{n_1}0, T^{n_2}0) \ar[dl] \ar[d] \ar[dr] & \\
(T^{n_1-1}0, T^{N-1}0) & (T^{N-1}0, \beta_{\pi^{-1}(i)}) & (\beta_{\pi^{-1}(i+1)-1}, T^{n_2-1}0)
}
\end{equation}

This instance contributes $|(T^{n_1 - 1}0, T^{N-1}0)| + L_{\pi^{-1}(i)} + R_{\pi^{-1}(i+1) - 1}$ to the set of gap lengths.

\item $d-3$ instances of

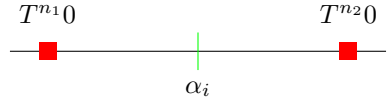
\begin{center}
\begin{tikzpicture}
\draw (-0.5,0) -- (4.5,0) node[fill=red] at (0,0) {} node at (0, 0.5) {$T^{n_1}0$} node at (2,-0.5) {$\alpha_i$} node[fill=red] at (4,0) {} node at (4,0.5) {$T^{n_2}0$};
\draw[green] (2,-0.25) -- (2, 0.25);
\end{tikzpicture}
\end{center}

\begin{equation}
\label{eq:IV3}
\tag{IV.3}
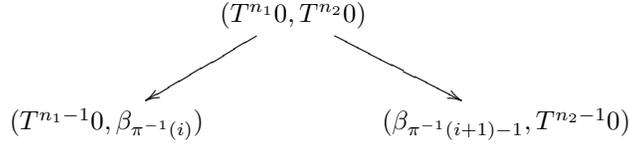
\xymatrix@C-2pc{
    & (T^{n_1}0, T^{n_2}0) \ar[dr] \ar[dl] & \\
(T^{n_1 - 1}0, \beta_{\pi^{-1}(i)}) &   & (\beta_{\pi^{-1}(i+1) - 1}, T^{n_2 - 1}0)
}
\end{equation}

Each of these instances contribute $L_{\pi^{-1}(i)} + R_{\pi^{-1}(i+1)-1}$ to the set of gap lengths, for a total of $d-3$ lengths.

\end{itemize}

\subsubsection{Case V}

\begin{itemize}

\item One instance of

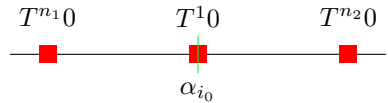
\begin{center}
\begin{tikzpicture}
\draw (-0.5,0) -- (4.5,0) node[fill=red] at (0,0) {} node at (0, 0.5) {$T^{n_1}0$} node[fill=red] at (2,0) {} node at (2,0.5) {$T^1 0$} node at (2,-0.5) {$\alpha_{i_0}$} node[fill=red] at (4,0) {} node at (4,0.5) {$T^{n_2}0$};
\draw[green] (2,-0.25) -- (2, 0.25);
\end{tikzpicture}
\end{center}

\begin{equation}
\label{eq:V1}
\tag{V.1}
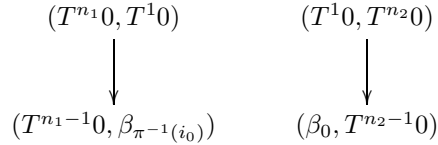
\xymatrix{
(T^{n_1}0, T^10) \ar[d]    & (T^1 0, T^{n_2}0) \ar[d] \\
(T^{n_1 - 1}0, \beta_{\pi^{-1}(i_0)})    & (\beta_0, T^{n_2 - 1} 0)
}
\end{equation}

This contributes the two lengths $L_{\pi^{-1}(i_0)}$ and $R_0$ to the set of gap lengths.

\item $d-2$ instances of

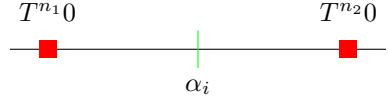
\begin{center}
\begin{tikzpicture}
\draw (-0.5,0) -- (4.5,0) node[fill=red] at (0,0) {}  node at (0, 0.5) {$T^{n_1}0$} node at (2,-0.5) {$\alpha_i$} node[fill=red] at (4,0) {} node at (4,0.5) {$T^{n_2}0$};
\draw[green] (2,-0.25) -- (2, 0.25);
\end{tikzpicture}
\end{center}

\begin{equation}
\label{eq:V2}
\tag{V.2}
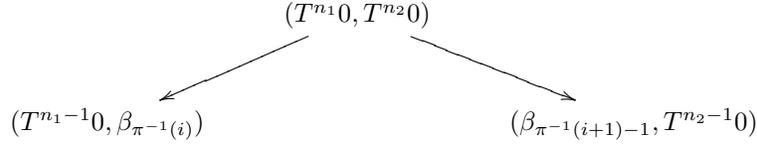
\xymatrix{
    &  (T^{n_1}0, T^{n_2}0) \ar[dl] \ar[dr]  &   \\
(T^{n_1 - 1}0, \beta_{\pi^{-1}(i)}) & & (\beta_{\pi^{-1}(i+1) - 1}, T^{n_2 - 1}0)
}
\end{equation}

Each of those instances contributes $L_{\pi^{-1}(i)} + R_{\pi^{-1}(i+1) - 1}$ to the set of gap lengths, for a total of $d-2$ elements.

\item One instance of

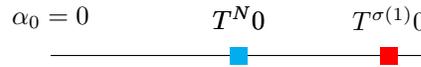
\begin{center}
\begin{tikzpicture}
\draw (-0.5,0) -- (4.5,0) node at (-0.5, 0.5) {$\alpha_0 = 0$} node[fill=cyan] at (2,0) {} node at (2, 0.5) {$T^N0$} node at (2,0.5) {$T^N 0$} node[fill=red] at (4,0) {} node at (4,0.5) {$T^{\sigma(1)}0$};
\end{tikzpicture}
\end{center}

This has the first gap length $R_0$.
\end{itemize}

\subsubsection{Case VI}

\begin{itemize}

\item One instance of

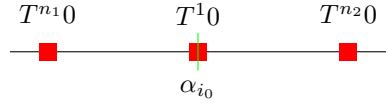
\begin{center}
\begin{tikzpicture}
\draw (-0.5,0) -- (4.5,0) node[fill=red] at (0,0) {} node at (0, 0.5) {$T^{n_1}0$} node[fill=red] at (2,0) {} node at (2,0.5) {$T^1 0$} node at (2,-0.5) {$\alpha_{i_0}$} node[fill=red] at (4,0) {} node at (4,0.5) {$T^{n_2}0$};
\draw[green] (2,-0.25) -- (2, 0.25);
\end{tikzpicture}
\end{center}

\begin{equation}
\label{eq:VI1}
\tag{VI.1}
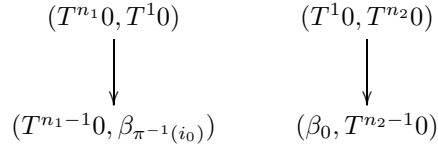
\xymatrix{
(T^{n_1}0, T^10) \ar[d]    & (T^1 0, T^{n_2}0) \ar[d] \\
(T^{n_1 - 1}0, \beta_{\pi^{-1}(i_0)})    & (\beta_0, T^{n_2 - 1} 0)
}
\end{equation}

This contributes the two lengths $L_{\pi^{-1}(i_0)}$ and $R_0$ to the set of gap lengths.

\item $d-2$ instances of

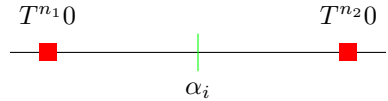
\begin{center}
\begin{tikzpicture}
\draw (-0.5,0) -- (4.5,0) node[fill=red] at (0,0) {}  node at (0, 0.5) {$T^{n_1}0$} node at (2,-0.5) {$\alpha_i$} node[fill=red] at (4,0) {} node at (4,0.5) {$T^{n_2}0$};
\draw[green] (2,-0.25) -- (2, 0.25);
\end{tikzpicture}
\end{center}

\begin{equation}
\label{eq:VI2}
\tag{VI.2}
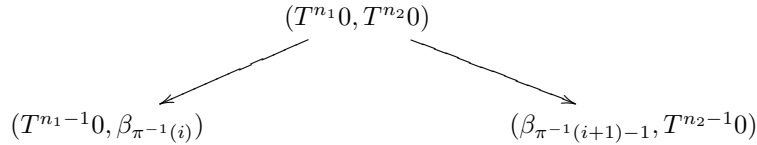
\xymatrix{
    &  (T^{n_1}0, T^{n_2}0) \ar[dl] \ar[dr]  &   \\
(T^{n_1 - 1}0, \beta_{\pi^{-1}(i)}) & & (\beta_{\pi^{-1}(i+1) - 1}, T^{n_2 - 1}0)
}
\end{equation}

Each of those instances contributes $L_{\pi^{-1}(i)} + R_{\pi^{-1}(i+1) - 1}$ to the set of gap lengths, for a total of $d-2$ elements.

\item One instance of

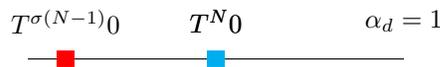
\begin{center}
\begin{tikzpicture}
\draw (-0.5,0) -- (4.5,0) node[fill=red] at (0,0) {} node at (0, 0.5) {$T^{\sigma(N-1)}0$} node[fill=cyan] at (2,0) {} node at (2, 0.5) {$T^N0$} node at (2,0.5) {$T^N 0$} node at (4.5,0.5) {$\alpha_d = 1$};
\end{tikzpicture}
\end{center}

This has the last gap length $L_d$.
\end{itemize}

\subsection{Gaps that Get Pulled to Other Gaps}
\label{subsection: gaps that get pulled to other gaps}

If a gap $(T^{n_1}0, T^{n_2}0)$ with $n_1, n_2 \neq 1$ does not include $T^N 0$ or a discontinuity $\alpha_i$ of $T^{-1}$, it will be pulled back by $T^{-1}$ to the gap $(T^{n_1 - 1}0, T^{n_2 - 1}0)$.

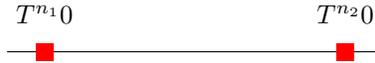
\begin{center}
\begin{tikzpicture}
\draw (-0.5,0) -- (4.5,0) node[fill=red] at (0,0) {} node at (0,0.5) {$T^{n_1}0$} node[fill=red] at (4,0) {} node at (4,0.5) {$T^{n_2}0$};
\end{tikzpicture}
\end{center}

The two gaps have the same lengths, so no instance of 
\begin{equation*}
\xymatrix{
(T^{n_1}0, T^{n_2}0) \ar[d] \\ (T^{n_1 - 1}0, T^{n_2 - 1}0)
}
\end{equation*}
contributes new elements to the set of gap lengths.

\subsection{The First and Last Gaps: $(0, T^{\sigma(1)}0)$ and $(T^{\sigma(N-1)}0, 1)$}
\label{subsection: the first and last gaps}

The first gap $(0, T^{\sigma(1)}0)$ shows up as a gap length in \ref{eq:I1}, \ref{eq:II1}, \ref{eq:IV1}, \ref{eq:V1}, and \ref{eq:VI1}; and so its length $R_0$ is accounted for in those cases. However, the length of the first gap $R_0$ shows up as a summand in \ref{eq:III1}. In that case, the first gap contributes $R_0$ to the set of gap lengths on its own.

The situation is similar with the last gap $(T^{\sigma(N-1)}0, 1)$ and its length $L_d$. We get one of two cases:
\begin{enumerate}
\item If $\pi^{-1}(i_0) = d$ (i.e. if $\pi(1) - \pi(d) = 1$) in cases I, III, IV, V, and VI above: the length of the last gap $L_d$ shows up as a gap length in equations \ref{eq:I1}, \ref{eq:III1}, \ref{eq:IV1}, \ref{eq:V1}, \ref{eq:VI1} and is accounted for. (Note that in case VI the last gap length shows as a gap length.)
\item Otherwise, the length of the last gap contribute as a summand in \ref{eq:I2}, \ref{eq:II1} or \ref{eq:II2}, \ref{eq:III2}, \ref{eq:IV2} or \ref{eq:IV3} and is not accounted for. In that case, the last gap contributes $L_d$ to the set of gap lengths on its own.
\end{enumerate}

Adding up all the contributions, we get the following table with upper bounds on the number of possible gap lengths:

\begin{center}
\begin{tabular}{|c|c|c|}
\hline
 & $\pi^{-1}(i_0) = d$ & $\pi^{-1}(i_0) \neq d$ \\
\hline
Case I & $d+1$  & $d+2$ \\
\hline
Case II & $d+1$  & $d+1$ \\
\hline
Case III & $d+1$  & $d+2$  \\
\hline
Case IV & $d$ & $d+1$ \\
\hline
Case V & $d$  & $d+1$ \\
\hline
Case VI & $d$  & $d+1$ \\
\hline
\end{tabular}
\end{center}
This proves \cref{theorem: d + 2 gap theorem}.

\end{document}